\documentclass[a4paper]{article}

\usepackage{amsmath,amssymb,amsthm}
\usepackage{graphicx}
\usepackage[english]{babel}
\usepackage[utf8]{inputenc}
\usepackage[T1]{fontenc}
\usepackage{mathtools}
\usepackage{bbm}
\usepackage[margin=3cm]{geometry}
\usepackage{xcolor}
\usepackage{enumitem}
\usepackage{hyperref}
\allowdisplaybreaks

\newcommand{\N}{\mathbb{N}}
\newcommand{\Z}{\mathbb{Z}}
\newcommand{\Q}{\mathbb{Q}}
\newcommand{\R}{\mathbb{R}}

\newcommand{\E}{\mathbb{E}}
\newcommand{\ud}{\,\mathrm{d}}

\newcommand{\e}{{\mathrm{e}}}

\newcommand{\PP}{\mathbf{P}}

\newcommand{\F}{\mathcal{F}}

\newcommand{\Sc}{\mathcal{S}}
\newcommand{\HH}{\mathcal{H}}

\newcommand{\q}{\mathfrak{q}}
\newcommand{\I}{\mathbbm{1}}

\newcommand{\Deltad}{{\Delta_{\mathsf{d}}}}

\DeclareMathOperator{\supp}{supp}
\DeclareMathOperator{\diam}{diam}
\DeclareMathOperator{\spec}{spec}
\DeclareMathOperator{\sgn}{sgn}

\DeclareMathOperator{\dista}{dist}

\DeclareMathOperator{\Cov}{Cov}

\DeclareMathOperator{\Id}{Id}
\numberwithin{equation}{section}
\numberwithin{figure}{section}

\newtheorem{theorem}{Theorem}[section]

\newtheorem{lemma}[theorem]{Lemma}

\newtheorem{remark}[theorem]{Remark}
\theoremstyle{definition}
\newtheorem{definition}[theorem]{Definition}

\begin{document}

\title{Finite range decompositions of Gaussian fields with applications to level-set percolation}
\author{Florian Schweiger\thanks{
Department of Mathematics, 
Weizmann Institute of Science,
Rehovot 76100, Israel.\newline
Current address: Section des mathématiques, Université de Genève, rue du Conseil-Général 7-9, 1205 Genève, Switzerland.\newline
Email: \texttt{florian.schweiger@unige.ch}}
}
\date{\today}
\maketitle
\begin{abstract}
In a recent work \cite{M22}, Muirhead has studied level-set percolation of (discrete or continuous) Gaussian fields, and has shown sharpness of the associated phase transition under the assumption that the field has a certain multiscale white noise decomposition, a variant of a finite-range decomposition.
We show that a large class of Gaussian fields have such a white noise decomposition with optimal decay parameter. Examples include the discrete Gaussian free field, the discrete membrane model, and the mollified continuous Gaussian free field. This answers various questions from \cite{M22}.

Our construction of the white-noise decomposition is a refinement of Bauerschmidt's construction of a finite-range decomposition \cite{B13}.
In the continuous setting our construction is very similar to Bauerschmidt's, while in the discrete setting several new ideas are needed, including the use of a result by P\'olya and Szeg\H{o} on polynomials that take positive values on the positive real line.

\end{abstract}

\paragraph{Keywords:} finite range decomposition, positive definite kernel, Gaussian field, level set percolation.

\paragraph{MSC subject codes (2020):} 60G15 (35J08, 82B43)

\section{Introduction and main results}
\subsection{Muirhead's work on level-set percolation}
Recently, Muirhead \cite{M22} has introduced a class of Gaussian fields for which there is a certain multiscale white noise decomposition. Let us reproduce his definition here, slightly rephrased:

\begin{definition}\label{d:Falpha}
Let $d\ge 2$, $\boldsymbol\alpha>0$. A continuous Gaussian field $f$ on $\R^d$ is in class $\F_{\boldsymbol\alpha}$ if we can write it as
\[f\overset{\text{law}}{=}\q*_1 W:=\int_{\R^d\times\R^+}\q(\cdot-y,t)\ud W(y,t)\]
for some standard white noise $W$ on $\R^d\times\R^+$ and some $\q\in L^2(\R^d\times\R^+)$ satisfying
\begin{itemize}
\item[(i)] (Symmetry) We have $\q(x,t)=\q(-x,t)$ for all $x,t$.
\item[(ii)] (Smoothness) We have $\q(\cdot,t)\in C^3(\R^d)$ for all $t$, and $\nabla_x^k\q\in L^2(\R^d\times\R^+)\cap L^\infty_{\text{loc}}(\R^d\times\R^+)$ for any $k\in\{0,1,2,3\}$.
\item[(iii)] (Non-degeneracy) There is $t_*\ge0$ such that $\q(\cdot,t_*)$ is not identically 0 and if $t_n\to t_*$ then $\q(\cdot,t_n)\to \q(\cdot,t_*)$ in $L^1(\R^d)$.
\item[(iv)] (Finite range) We have $\supp \q\subset\left\{(x,t)\colon |x|\le\frac t2\right\}$.
\item[(v)] (Decay) There is a constant $C$ such that
\[\int_{\R^d\times[t,\infty)}\q^2(x,s)\ud x\ud s\le\frac{C}{t^{\boldsymbol\alpha}}.\]
\end{itemize}
A discrete Gaussian field $f$ on $\Z^d$ is in class $\F_{\boldsymbol\alpha}$ if it is the restriction of a continuous field in $\F_{\boldsymbol\alpha}$ to $\Z^d$, or more precisely, if we can write it as
\[f\overset{\text{law}}{=}\q*_1 W\restriction_{\Z^d}:=\int_{\R^d\times\R^+}\q(\cdot-y,t)\ud W(y,t)\restriction_{\Z^d}\]
for some standard white noise $W$ on $\R^d\times\R^+$ and some $\q\in L^2(\R^d\times\R^+)$ satisfying (i), (iii), (iv), (v) (but not necessarily (ii)).
\end{definition}

In the following we will call a decomposition as in Definition \ref{d:Falpha} a white noise decomposition (to distinguish it from other kinds of finite-range decompositions).

Muirhead's motivation is to study level-set percolation of (discrete or continuous) Gaussian fields. Namely, given a continuous Gaussian field $f$, one considers a cut-off $\ell$ and studies the percolative properties of the set $\{x\colon f(x)\ge -\ell\}$. One defines the critical level $\ell_c$ as the infimum of all $\ell$ such that the probability of 0 being in an unbounded cluster in $\{x\colon f(x)\ge -\ell\}$ is positive. For most fields of interest, $\ell_c\in(-\infty,\infty)$, i.e. there is a nontrivial phase transition. It is a natural question whether this phase-transition is sharp.

Muirhead's main result is on level-set percolation of fields in $\F:=\bigcup_{{\boldsymbol\alpha}>0}\F_{\boldsymbol\alpha}$, and in particular he proves sharpness of the phase transition for level-set percolation of any field in $\F$, with decay parameter depending on ${\boldsymbol\alpha}$. More precisely, the result \cite[Theorem 1.2]{M22} states that for any (discrete or continuous) field in $\F_{\boldsymbol\alpha}$ for any subcritical level the connection probabilities decay stretched-exponentially with any exponent in $[0,{\boldsymbol\alpha})\cap[0,1]$, while for any supercritical level the density of the infinite cluster grows polynomially with any rate in $\left(\frac{2(d-1)}{{\boldsymbol\alpha}},\infty\right)\cap[1,\infty)$.

Level set percolation had previously been studied in the case of the discrete Gaussian free field (see for example \cite{RS13, DPR18}), and in the case of the (discrete) membrane model \cite{CN21}. A continuous field of great interest is given by the continuous Gaussian free field, mollified by convolution with a sufficiently smooth kernel (see Section \ref{s:fields} for a reminder of how these fields are defined).
For the discrete Gaussian free field, sharpness of the phase transition had recently been shown in a break-through result in \cite{DGRS20}). For the other two models, sharpness had been an open question.

This lead to the natural question whether Muirhead's result applies to these fields, or in other words whether they are in $\F$ and what the optimal parameter ${\boldsymbol\alpha}$ is. Muirhead proved that the discrete Gaussian free field is in $\F_{(d-2)/2}$ and that the discrete membrane model is in $\F_{(d-4)/2}$. In particular, this established for the first time sharpness of the phase transition for the membrane model in $d\ge5$. However, as Muirhead himself remarks, based on the decay of correlations of the respective fields these results are not optimal, and in fact he conjectures that the discrete Gaussian free field is in $\F_{d-2}$, and the membrane model in $\F_{d-4}$. As for the mollified continuous Gaussian free field, he poses as an open question whether it is in $\F$.

The main goal of the present paper is to resolve these questions positively.

\begin{theorem}\label{t:fields_in_Falpha}
The discrete Gaussian free field in dimension $d\ge3$ and the membrane model in dimension $d\ge5$ are in classes $\F_{d-2}$ and $\F_{d-4}$, respectively. Moreover, the continuous Gaussian free field in dimension $d\ge3$, mollified by some $k\in C^3_c(\R^d)$ which is symmetric around 0, is in class $\F_{d-2}$, and the continuous membrane model in dimension $d\ge5$, mollified by some $k$ as above, is in $\F_{d-4}$.
\end{theorem}
In particular, Muirhead's result applies to all these fields and allows to conclude that the phase transition is sharp. For the two continous field, this sharpness is a new result.

Our results are actually far more general than Theorem \ref{t:fields_in_Falpha}. In fact we provide a general framework that allows to find white noise decompositions as in Definition \ref{d:Falpha} for wide classes of Gaussian fields (see the discussion in Section \ref{s:discussion} for further examples). Before introducing this framework and motivating our proof, let us discuss finite-range decompositions more generally.

\subsection{Finite range decompositions}

\paragraph{Context and previous results}
A Gaussian field $f$ on a metric measure space $(X,d,\mu)$ (where in our setting either $X=\Z^d$ or $X=\R^d$) corresponds to a positive semidefinite quadratic form $\Phi(u,v)=\Cov\left((u,f)_{L^2(X)},(v,f)_{L^2(X)}\right)$ for $u,v\in C_c(X)$, say. Often this quadratic form has infinite range, and this is a serious obstacle to the mathematical analysis of such Gaussian fields. In particular, in mathematical implementations of the renormalization group method, one typically attempts to analyze the field scale by scale. For that purpose one needs a decomposition of $\Phi$ as a sum (or integral) of positive quadratic forms $\Phi_t$, where $t>0$ is a scale parameter. This decomposition then allows to write the field $f$ as a sum (or integral) of independent Gaussian fields $f_t$, which one can then analyze sequentially.

Of course, the $\Phi_t$ should be as simple as possible, and in fact it is often advantageous if the $\Phi_t$ have finite range, that is if $\Phi_t(u,v)=0$ when $\dista(\supp u,\supp v)>\theta(t)$ for a function $\theta\colon\R^+\to\R^+$.

The desire to use rigorous versions of the renormalization group method for random fields has led to the development of various approaches to find finite-range decompositions for given quadratic forms \cite{BGM04,BT06,AKM13,B13,B18}. Here the former three works construct finite range decompositions by averaging projections over well-chosen subspaces. The latter two works use a very different approach, introduced by Bauerschmidt: They define the decomposition via functional calculus and rely on the finite speed of propagation of the wave equation to ensure the finite range property.

In the mentioned works, the goal was to find a decomposition
\begin{equation}\label{e:finite_range}
\Phi(u,v)=\int_0^\infty\Phi_t(u,v)\ud t
\end{equation}
where the $\Phi_t$ are positive semidefinite and have finite range. 
As will become clear in a second, Muirhead's white noise decomposition requires more. Namely the Gaussian field $\q*_1 W$ has covariance
\begin{align*}\Cov\left((u,\q*_1 W)_{L^2(X)},(v,\q*_1 W)_{L^2(X)}\right)&=\int_{\R^d\times\R^+}\int_{\R^d}\int_{\R^d}u(x)\q(x-z,t)\q(z-y,t)v(y)\ud z\ud x\ud y\ud t\\
&=\int_0^\infty ( Q_t(u), Q_t(v))_{L^2(X)}\ud t
\end{align*}
with the operator $Q_t\colon D(Q)\subset D(Q_t)\to L^2(X)$ on $L^2(X)$ defined by
\[Q_t(u):=\q(\cdot,t)*_1 u=\int_{\R^d}\q(\cdot-x,t)u(x)\ud x.\]
So if a Gaussian field $f$ is in $\F$, then its covariance $\Phi(u,v)=\Cov\left((u,f)_{L^2(X)},(v,f)_{L^2(X)}\right)$ has a decomposition
\begin{equation}\label{e:finite_range_new}
\Phi(u,v)=\int_0^\infty (Q_t(u),Q_t(v))_{L^2(X)}\ud t
\end{equation}
where now the map $Q_t$ has finite range $\frac{t}{2}$ in the sense that $\supp Q_t(u)\subset N_{t/2}(\supp u)$ (where $N_r(A)=\{x\in X\colon\dista(x,A)<r\}$ is the $r$-neighborhood of $A\subset X$).

From the white noise decomposition \eqref{e:finite_range_new}, one can easily recover the classical decomposition \eqref{e:finite_range} by choosing $\Phi_t(u,v):=(Q_t(u),Q_t(v))_{L^2(X)}$. The converse need not hold, as in general, given a positive symmetric finite-range quadratic form $\Phi_t$ it is not clear whether one can represent it as $(\tilde Q_t(\cdot),\tilde Q_t(\cdot))_{L^2(X)}$ for some finite-range operator $\tilde Q_t$ (cf. Section \ref{s:discussion} for more on this).
Thus Muirhead's requirements on a white noise decomposition are genuinely stronger than those in the classical finite-range decompositions in the literature.

Decompositions like \eqref{e:finite_range_new} have appeared in \cite{HS02} (extending earlier work in \cite{FL86}), where it was shown that for each radial kernel on $\R^d$ such a decomposition exists when one additionally introduces a (possibly negative) weight, and there are sufficient conditions when that weight is non-negative. From this result one can deduce the existence of a white noise decomposition as in Definition \ref{d:Falpha} in some special cases (cf. Section \ref{s:discussion}).

\paragraph{Muirhead's approach}

In his proof that the discrete Gaussian free field and the membrane model are in $\F$, Muirhead relies on an argument from \cite{DGRS20} that uses the heat kernel to construct a white noise decomposition. Let us briefly sketch that argument in the case of the discrete Gaussian free field. There the Green's function $G(x):=G(0,x):=\Cov(f(0),f(y))$ can be written as $G(x)=\sum_{n=0}^\infty\PP(X^{\Z^d}_n=y)$, where $X_n^{\Z^d}$ is a simple random walk on $\Z^d$ starting at $0$. One now rewrites the right-hand side using simple random walk on the graph $\mathbb{M}^d$ that arises from $\Z^d$ by adding a vertex at the midpoint of each (nearest-neighbor) edge of $\Z^d$, and finds that
\[G(x)=\frac12\sum_{n=0}^\infty\PP(X^{\mathbb{M}^d}_n=y)\PP(X^{\mathbb{M}^d}_n=x-y).\]
This is a representation of $G$ as the sum over $n$ of self-convolutions of some functions. From it we quickly obtain a decomposition like \eqref{e:finite_range_new} (with a sum over $n$ instead of an integral over $t$). This decomposition has the finite-range property because the random walk $X^{\mathbb{M}^d}$ can move at most $n$ steps in time $n$. Using this, one can verify that the discrete Gaussian field is indeed in $\F_{(d-2)/2}$ (see \cite[Section 2.3.1]{M22} for details).

However, it is also easy to see the short-coming of this argument: The finite range property was deduced from the fact that simple random walk in discrete time stays in a set of diameter $n$ after $n$ steps. However, with high probability the walk will have moved only distance $C\sqrt{n}$ after $n$ steps. The loss of a factor 2 in the exponent here is the reason why the argument only shows that the field is in $\F_{(d-2)/2}$ and not, as expected, in $\F_{d-2}$.
Moreover, the approach breaks down completely when applied to the mollified continuous Gaussian free field, as continuous-time simple random walk is not deterministically bounded at any positive time.

\paragraph{Our approach}

We will use a different approach in the following. Namely we will modify Bauerschmidt's finite-range decomposition \cite{B13} so that it yields existence not just of a finite-range $\Phi_t$, but also existence of finite-range $Q_t$, at least if we make some small additional assumptions in some cases. Let us very briefly explain Bauerschmidt's approach and our new ideas, focusing on the case of the (discrete or continuous) Gaussian free field for simplicity. Let $\Deltad$ and $\Delta$ denote the discrete or continuous Laplacian, respectively, and let $L=-\Deltad$ or $L=-\Delta$.

Bauerschmidt constructs his decomposition by defining a family of non-negative functions $w_t\in C^\infty(I)$ for some interval $I$ containing the spectrum of $L$ such that
\begin{equation}\label{e:identity_lambda}
\frac{1}{\lambda}=\int_0^\infty tw_t(\lambda)\ud t.
\end{equation}
These functions give rise to a decomposition
\[(u,L^{-1}v)_{L^2(X)}=\int_0^\infty (u,tw_t(L)v)_{L^2(X)}\]
where $w_t(L)$ is defined via the spectral theorem and functional calculus. Then the quadratic form $\Phi_t(u,v):=(u,tw_t(L)v)_{L^2(X)}$ is positive semidefinite symmetric, and to obtain the decomposition \eqref{e:finite_range} it remains to ensure that it has finite range. Here Bauerschmidt's ingenuous idea is to use the finite speed of propagation of the (discrete or continuous) wave equation and to pick the $w_t$ in such a way that $w_t(L)$ is a suitable combination of the solution operators of the wave equation.

Now, the basic idea of our construction of a white noise decomposition is to choose functions $\tilde w_t\in C^\infty(I)$ (not necessarily non-negative) such that
\begin{equation}\label{e:identity_tildelambda}
\frac{1}{\lambda}=\int_0^\infty t|\tilde w_t(\lambda)|^2\ud t.
\end{equation}
Then we have
\[(u,L^{-1}v)_{L^2(X)}=\int_0^\infty (\sqrt{t}\tilde w_t(L)u,\sqrt{t}\tilde w_t(L))_{L^2(X)}\]
and so we can set $Q_t(u)=\sqrt{t}\tilde w_t(L)u$, provided that we can pick $\tilde w_t$ in such a way that $Q_t$ has finite range. It turns out that in the continuous case (i.e. if $X=\R^d$ and $L=-\Delta$) a very minor modification of Bauerschmidt's $w_t$ already has the required properties. 

The situation is less easy in the discrete case (i.e. if $X=\Z^d$). In that case we have been unable to construct $\tilde w_t$ in such a way that \eqref{e:identity_tildelambda} holds and $\tilde w_t$ has finite range (see the beginning of Section \ref{s:exPstar} for an explanation where the problem is). 

One way to solve this problem is to look for a more general white noise decomposition. Namely, instead of \eqref{e:finite_range_new}, we allow the $Q_t$ to take values in a separable Hilbert space $\HH$. That is, we consider operators $Q_t\colon D(Q)\subset D(Q_t)\to L^2(X,\HH)$ that satisfy
\begin{equation}\label{e:finite_range_new2}
\Phi(u,v)=\int_0^\infty (Q_t(u),Q_t(v))_{L^2(X,\HH)}\ud t
\end{equation}
and have finite range $\frac{t}{2}$. 
Now at first glance it might not be clear how to recover a white noise decomposition as in Definition \ref{d:Falpha} from this. However, there is an easy trick that solves this problem. Namely, in our applications $\HH$ will be equal to $\R^K$ for some $K\in\N$, and we will accommodate the extra $K$ dimensions in the codomain of $Q_t$ by letting $\q(\cdot,t)$ cycle through the $K$ components of $Q_t$.

The point now is that \eqref{e:finite_range_new2} can be obtained by writing $w_t(\lambda)$ in \eqref{e:identity_lambda} as a sum of several squares (not just the one square $|\tilde w_t(\lambda)|^2$ in \eqref{e:identity_tildelambda}). 
So instead of a representation as in \eqref{e:identity_tildelambda} with one square, we look for representations with a fixed bounded number of squares.
To find it, we go back to Bauerschmidt's construction of $w_t$ in \eqref{e:identity_lambda}. In the discrete setting, the constructed $w_t$ is a polynomial in $\lambda$ of degree at most $t$, and ideally we would want this polynomial to be a sum of squares of polynomials. This is not quite possible, as $w_t(\lambda)$ in general takes negative values, and so it cannot be a sum of squares.

However, we can arrange things so that $w_t$ is non-negative for $\lambda\in(-\infty,2B]$ (where $B$ is a constant such that the spectrum of $-\Deltad$ is contained in $[0,B]$. Now by a classical result of P\'olya and Szeg\H{o} \cite{PS98}, there are real polynomials $a_t^{(j)}$ for $j\in\{1,2,3,4\}$ such that \[
w_t(\lambda)=(a_t^{(1)}(\lambda))^2+(a_t^{(2)}(\lambda))^2+(2B-\lambda)\left((a_t^{(3)}(\lambda))^2+(a_t^{(4)}(\lambda))^2\right).\]
This is almost a representation as sums of squares, just the factor $2B-\lambda$ is problematic. However, if $B\ge4d$, then one can easily check that the operator $2B \Id+\Deltad$ can be written as $R^*R$, where $R\colon L^2(X)\to L^2(X,\R^{d+1})$ is given by $(Ru(x))_1=\sqrt{2B-4d}\,u(x)$ and $(Ru(x))_j=u(x+e_{j-1})+u(x)$ for $2\le j\le d+1$. So, while we can not make the factor $2B-\lambda$ be a square, we can nonetheless represent $2B\Id+\Deltad$ as the product of a finite-range operator with its adjoint, and this is good enough for our purposes.
We can now rewrite \eqref{e:identity_lambda} as
\begin{align*}
&(u,(-\Deltad)^{-1}v)_{L^2(X)}\\
&=\int_0^\infty t\sum_{j=1}^2\left( u,(a_t^{(j)}(-\Deltad))^2v\right)_{L^2(X)}+\sum_{j=3}^4t\left( u,(2B\Id+\Deltad)(a_t^{(j)}(-\Deltad))^2v\right)_{L^2(X)}\ud t\\
&=\int_0^\infty \sum_{j=1}^2\left(\sqrt{t} a_t^{(j)}(-\Deltad)u,\sqrt{t} a_t^{(j)}(-\Deltad)v\right)_{L^2(X)}\\
&\qquad+\sum_{j=3}^4\left(\sqrt{t} R a_t^{(j)}(-\Deltad)u,\sqrt{t} R a_t^{(j)}(-\Deltad)v\right)_{L^2(X,\R^{d+1})}\ud t
\end{align*}
where we used that $R^*R$ and $-\Deltad$ commute. The integrand here is a sum of $1+1+(d+1)+(d+1)=2d+4$ scalar products in $L^2(X)$, and so it directly leads to a decomposition as in \eqref{e:finite_range_new2} with $\HH=\R^{2d+4}$.

If we go beyond the special case of the discrete Laplacian, the approach that we just sketched continues to apply. In fact, our construction of a white noise decomposition holds in a general setting very similar to the one in \cite{B13}. The main change is that, given a generator $L$, in the discrete case we assume a priori the existence of $B>0$ and $R\colon L^2(X)\to L^2(X,\HH_0)$ for some Hilbert space $\HH_0$ such that we can factorize $2B\Id-L=R^*R$. Under this assumption we then obtain a finite-range decomposition as in \eqref{e:finite_range_new2} with $\HH=\R^2\times \HH_0^2$.

\subsection{Main result}
In order to state our main result on finite-range decompositions in detail, we need to introduce some definitions. We work in the same setting as in \cite{B13}, and so we begin by reviewing it.

We let $(X,d,\mu)$ be a metric measure space, and let $E\colon D(E)\times D(E)\to\R$ be a regular closed symmetric form on $L^2(X)$ and $L\colon D(L)\to L^2(X)$ be its generator. $L$ is self-adjoint, and by the spectral theorem there exists a projection-valued spectral measure $P$ so that $L=\int_0^\infty\lambda\ud P_\lambda$, and for each Borel-measurable $F\colon[0,\infty)\to\R$ we can define the self-adjoint operator $F(L)=\int_0^\infty F(\lambda)\ud P_\lambda$.

Note that we can then write \[E(u,v)=(L^{1/2}u,L^{1/2}v)_{L^2(X)}\] for $u,v\in D(E)=D(L^{1/2})$. We can also easily define the dual form \[\Phi(u,v)=(L^{-1/2}u,L^{-1/2}v)_{L^2(X)}\] for $u,v\in D(\Phi)=D(L^{-1/2})$.

Similarly as in \cite{B13}, we consider the following two finite propagation speed conditions
\begin{equation}\tag{$P_{\gamma,\theta}$}\label{e:P}
\supp(\cos(L^{\gamma/2}t)u)\subset N_{\theta(t)}(\supp(u))\quad\forall u\in C_c(X), t>0
\end{equation}
and
\begin{equation}\tag{$P^*_{\gamma,B,\theta}$}\label{e:Pstar}
\begin{split}
E(u,u)\le B\|u\|_{L^2(X)}\quad\forall u\in L^2(X),\\
\supp(L^{\gamma n} u)\subset N_{\theta(n)}(\supp(u))\quad\forall u\in C_c(X), n\in\N
\end{split}
\end{equation}
for $\gamma>0$, $B>0$ and $\theta\colon\R^+\to\R^+$, and we remind that $N_r(A)=\{x\in X\colon\dista(x,A)<r\}$. The former condition is taken directly from \cite{B13}, the latter appears there for $\gamma=1$ only.

We also consider a further assumption as in \cite{B13}, namely a heat kernel bound. To state it, let $p_t$ be the kernel of the semigroup $(\e^{-tL})_{t>0}$. We consider the condition
\begin{equation}\tag{$H_{\alpha,\omega}$}\label{e:H}
p_t(x,x)\le \frac{\omega(x)}{t^{\alpha/2}}\quad\forall x\in X,t>0
\end{equation}
for $\alpha>0$ and $\omega\colon X\to\R^+$ bounded (and we implicitly assume the existence of the continuous kernel $p_t$).

Finally we consider two new assumptions. The first one states that $(2B)^\gamma\Id-L^\gamma$ can be factorized into the product of some finite-range linear operator with its adjoint. That is, we consider the condition \begin{equation}\tag{$F_{\gamma,B,\theta_R,R,\HH_0}$}\label{e:F}
\begin{split}
(2B)^\gamma(u,v)_{L^2(X)}-(u,L^\gamma v)=(Ru,Rv)_{L^2(X,\HH_0)}\quad\forall u,v\in L^2(X)\\\supp(Ru)\subset N_{\theta_R}(\supp u)\quad\forall u\in C_c(X)
\end{split}
\end{equation}
for some Hilbert space $\HH_0$, some $R\colon L^2(X)\to L^2(X,\HH_0)$ and some $\theta_R>0$. The second one is a regularity assumption for $R$, namely that there is some $l_R\ge0$ such that the operator $R(1+L)^{-l_R}$ is given by convolution with a continuous kernel $r\colon X\times X\to\HH_0$ where
\begin{equation}\tag{$H^*_{l_R,R,\HH_0}$}\label{e:R}
x\mapsto r(x,\cdot)\in C_b(X,L^1(X,\HH_0)).
\end{equation}

We postpone a discussion of these assumptions to the next section. For now, we discuss what can be proven assuming them. The main result of \cite{B13} is that under assumption \eqref{e:P} or \eqref{e:Pstar} there is a finite range decomposition as in \eqref{e:finite_range} with $L^2$-bounds on $\Phi_t$, and pointwise bounds when one additionally assumes \eqref{e:H}. We generalize his result in the following ways: Our first result is that under the same assumption \eqref{e:P} we even have a white noise decomposition as in \eqref{e:finite_range_new2} with $\HH=\R$. Our second result is that if we assume \eqref{e:Pstar} and additionally \eqref{e:F}, then there is a white noise decomposition as in \eqref{e:finite_range_new2} where the operators $Q_t$ take values in $\HH:=\R^2\times\HH_0^2$. Note that Bauerschmidt's result in the discrete case is restricted to $\gamma=1$, while we allow general $\gamma$ such that $\frac{1}{\gamma}\in\N$.
In both cases, we have $L^2$-bounds on $Q_t$ and pointwise bounds if we additionally assume \eqref{e:H} and \eqref{e:R}.

\begin{theorem}\label{t:mainthm}
Let $(X,d,\mu)$ be a metric measure space, $E$ be a regular closed symmetric positive semidefinite form on $L^2(X)$ and $L$ be its generator. 
\begin{itemize}
\item[(i)] Suppose that $\gamma>0$ and that assumption \eqref{e:P} holds. Then there exist a family of linear maps $Q_t\colon L^2(X)\to L^2(X)$ indexed by $t\in\R^+$ such that $Q_t$ has range at most $\theta(t)$, the operator norm of $Q_t$ is bounded by $C_{\gamma} t^{(2-\gamma)/(2\gamma)}$, and such that we have the white noise decomposition
\begin{equation}\label{e:finite_range_mainthm}
\Phi(u,v)=\int_0^\infty (Q_t(u),Q_t(v))_{L^2(X)}\ud t.
\end{equation}
Moreover, if we additionally assume \eqref{e:H} then $Q_t$ is given by convolution with a continuous kernel $q_t\colon X\times X$ that satisfies the bound
\begin{align}
|q_t(x,y)|&\le \frac{C_{\alpha,\gamma}\sqrt{\omega(x)\omega(y)}}{t^{(2\alpha+\gamma-2)/(2\gamma)}},\label{e:estimate_kernel}\\
\|q_t(x,\cdot)\|_{L^2(X)}&\le \frac{C_{\alpha,\gamma}\sqrt{\omega(x)}}{t^{(\alpha+\gamma-2)/(2\gamma)}}.\label{e:estimate_kernel2}
\end{align}
\item[(ii)] Suppose instead that $\frac1\gamma\in\N$ and that assumptions \eqref{e:Pstar} and \eqref{e:F} both hold. Let $\HH=\R^2\times\HH_0^2$.
Then there exist a family of linear maps $Q_t\colon L^2(X)\to L^2(X,\HH)$ indexed by $t\in\R^+$ such that $Q_t$ has range at most $\max(\theta(t),\theta(t-1)+\theta_R)$, the operator norm of $Q_t$ is bounded by $C_{\gamma,B} t^{(2-\gamma)/(2\gamma)}$, and such that we have the finite-range decomposition
\begin{equation}\label{e:finite_range_mainthmH}
\Phi(u,v)=\int_0^\infty (Q_t(u),Q_t(v))_{L^2(X,\HH)}\ud t.
\end{equation}
Moreover, if we additionally assume \eqref{e:H} and \eqref{e:R}, then $Q_t$ is given by convolution with a continuous kernel $q_t\colon X\times X\to\HH$ that satisfies the bounds
\begin{align}
|q_t(x,y)|_{\HH}&\le \frac{C_{\alpha,\gamma,B,l_R,R}\sup_z\omega(z)}{t^{(2\alpha+\gamma-2)/(2\gamma)}},\label{e:estimate_kernelH}\\
\|q_t(x,\cdot)\|_{L^2(X,\HH)}&\le \frac{C_{\alpha,\gamma,B}\sqrt{\omega(x)}}{t^{(\alpha+\gamma-2)/(2\gamma)}}\label{e:estimate_kernelH2}
\end{align}
for $t\ge1$, while for $t<1$ we have the explicit formula
\begin{equation}\label{e:formula_kernelH}
q_t(x,y)=F_{\gamma,B}(t)t^{(1-\gamma)/\gamma}\I_{x=y}e_1
\end{equation}
for some bounded function $F_{\gamma,B}\colon[0,1]\to[0,\infty)$, where $e_1=(1,0,0,0)\in\HH=\R^2\times\HH_0^2$.
\end{itemize}
\end{theorem}

\subsection{Further discussion}\label{s:discussion}
\paragraph{Comparison to earlier works}
Let us compare Theorem \ref{t:mainthm} with the results in \cite{B13,HS02}, beginning with the former. In the continuous case (part (i)) our assumptions are the same as in \cite[Theorem 1.1]{B13}. If in the setting of our theorem we define $\Phi_t(u,v):=(Q_t(u),Q_t(v))_{L^2(X,\HH)}$, then the $\Phi_t$ have range $2\theta(t)$ and satisfy the other assumptions of \cite[Theorem 1.1]{B13}. Thus, apart from the minor detail that the range is now $2\theta(t)$ instead of $\theta(t)$, our result implies Bauerschmidt's in this case.

In the discrete case (part (ii)) we need to make the assumptions \eqref{e:F} and \eqref{e:R} that are not needed in \cite[Theorem 1.1]{B13}. Under these stronger assumptions, our result again implies a version of Bauerschmidt's. Let us emphasize, though, that we allow arbitrary $\gamma$ with $\frac{1}{\gamma}\in\N$, not just $\gamma=1$. The argument that we use to cover the case $\frac{1}{\gamma}\in\{2,3,\ldots\}$, however, could probably also be applied in Bauerschmidt's setting, so that with some additional work one should be able to extend \cite[Theorem 1.1]{B13} also to this case.

Some special cases of part (i) of Theorem \ref{t:mainthm} can also be deduced from \cite[Theorem 1]{HS02}. There for radially symmetric translation-invariant covariance kernels on $\R^d$ under weak assumptions a decomposition like \eqref{e:finite_range_mainthm}, but with an additional radial weight is constructed. Whenever that weight is non-negative this yields a white noise decomposition as in \eqref{e:finite_range_mainthm} itself. Now, as remarked in \cite[Example 3]{HS02} the radial weight need not be non-negative even if the kernel is positive semidefinite. However, there is an explicit formula for the radial weight that allows to verify that for power-law kernels the radial weight is non-negative. Thus, for the Laplacian on $\R^d$, for example, Theorem \ref{t:mainthm} already follows from the results in \cite{HS02}. Nonetheless, our result is clearly far more general.

\paragraph{Discussion of the assumptions and further examples}
In setting (i), our assumptions \eqref{e:P} and \eqref{e:H} are the same as in \cite{B13}. As discussed there, the Dirichlet form associated with any even-order uniformly elliptic differential operator on $\R^d$ with possibly variable coefficients satisfies the assumptions. The same holds true for (possibly fractional) powers $L^s$ of such operators (by choosing $\gamma=\frac{1}{s}$). This means that Theorem \ref{t:fields_in_Falpha} applies to all these operators. By similar arguments as in the proof of Theorem \ref{t:fields_in_Falpha}, the corresponding Gaussian fields (suitably mollified) are in $\F$. For example, consider the fractional Gaussian fields with Hamiltonian $(u,(-\Delta)^su)$ for $s>0$ (see \cite{LSSW16} for background). Along the lines of the proof of Theorem \ref{t:fields_in_Falpha} one can also show that in the supercritical regime $s<\frac{d}{2}$ these fields, suitably mollified, are in $\F_{d-2s}$. 

In setting (ii), the assumptions \eqref{e:Pstar} and \eqref{e:H} are like the one in \cite{B13}. As discussed in \cite{B13}, every even-order finite-difference operator on $\Z^d$ satisfies them. As we have the freedom to choose $\gamma$ subject to $\frac{1}{\gamma}\in\N$, we can also take integer powers of such operators while maintaining the optimal decay rates.

Importantly, however, in setting (ii) we need the additional assumptions \eqref{e:F} and \eqref{e:R}. These assumptions are fairly mild, though. Indeed, for example, any operator of the form $Lu(x)=\sum_{y\in\Z^d}\omega_{xy}u(y)$, where we only require that $\omega_{xy}=\omega_{yx}$, $\omega_{xy}=0$ if $|x-y|\ge \Theta$ for some $\Theta\in\N$ and $\sup_{x,y\in\Z^d}|\omega_{xy}|<\infty$, satisfies \eqref{e:F} and \eqref{e:R} for $\gamma=1$, and $\HH_0=\R^K$ for some sufficiently large $K$. To see this, let $S=\sup_{x,y\in\Z^d}|\omega_{xy}|<\infty$, and take $B=(2\Theta+1)^dS$. By polarization, it suffices to check \eqref{e:F} for $u=v$. We can now write
\begin{align*}
&2B(u,u)_{L^2(X)}-E(u,u)\\
&=2B\sum_{x\in\Z^d}u(x)^2-\sum_{x\in\Z^d}\sum_{y\in\Z^d}\omega_{xy}u(x)u(y)\\
&=\sum_{x\in\Z^d}\left(\left(2(2\Theta+1)^dS-\sum_{y\in\Z^d}|\omega_{xy}|\right)u(x)^2+\sum_{y\in\Z^d}|\omega_{xy}|\left(\frac{u(x)^2+u(y)^2}{2}-\sgn(\omega_{xy})u(x)u(y)\right)\right)\\
&=\sum_{x\in\Z^d}\left(\Gamma_{xy}u(x)^2+\sum_{y\in\Z^d}\frac{|\omega_{xy}|}{2}\left(u(x)-\sgn(\omega_{xy})u(y)\right)^2\right)
\end{align*}
where $\Gamma_{xy}:=2(2\Theta+1)^dS-\sum_{y\in\Z^d}|\omega_{xy}|\ge0$ by our assumptions on $\omega$.

From this formula we can read off that if we choose $\HH_0=\R^{1+(2d+1)^d}$ and
\[Ru(x)=\left(\sqrt{\Gamma_{xy}}u(x),\left(\sqrt{\frac{|\omega_{xy}|}{2}}\left(u(x)-\sgn(\omega_{xy})u(y)\right)\right)_{y\in x+[-\Theta,\Theta]^d\cap\Z^d}\right)\]
then assumption \eqref{e:F} is satisfied, and \eqref{e:R} holds as well (even with $l_R=0$).

Thus, assumptions \eqref{e:F} and \eqref{e:R} hold for many discrete operators $L$ (even many that are not positive semidefinite). In particular, all even order uniformly elliptic finite-difference operators with possibly variable, but bounded coefficients satisfy all assumptions of Theorem \ref{t:mainthm}. Again, following the lines of the proof of Theorem \ref{t:fields_in_Falpha}, one can show that the corresponding Gaussian fields in supercritial dimension are in $\F$. 

However, note that in the discrete setting we needed to assume $\frac{1}{\gamma}\in\N$ in Theorem \ref{t:mainthm}. The result might well be true without this assumption, but it is essential for our proof. So, currently we do not know whether discrete fractional Gaussian fields (as discussed in \cite{DNS22}) are in $\F$.
\paragraph{Possible extensions and variants}
The estimates in Theorem \ref{t:mainthm} should all have the optimal dependence on $t$. If $X=\R^d$ and $L$ is an even-order constant-coefficient elliptic differential operator, then in setting (i) our construction of $Q_t$ allows to straightforwardly obtain smoothness of $q_t(x,y)$ as well as quantitative estimates (with optimal decay rate) on its derivatives by arguing in Fourier space. We refer to \cite[Section 3.2]{B13} for a discussion of such estimates in this general setting.
The same is true in principle in setting (ii), however our approach currently only yields sub-optimal estimates on the derivatives of $q_t$ in this case. The technical reason for this is that our construction of the functions $a_t^{(j)}(\lambda)$ in Lemma \ref{l:polyaszego_cont} uses the result by P\'olya and Szeg\H o, which makes it quite non-explicit and in particular makes it difficult to obtain estimates for their derivatives in $\lambda$.
As is explained in Remark \ref{r:markov}, one can obtain non-optimal estimates for these derivatives using the Markov brothers' inequality \cite{MG16}, and this yields non-optimal estimates for the derivatives of $q_t$.

Let us point out that it is clear from our construction that $Q_t$ and $q_t$ inherit properties from $L$. In particular, if $L$ and $R$ are translation-invariant, then $Q_t$ and $q_t$ will be translation-invariant as well.

If $\HH=\R^K$ is finite-dimensional, then one can also obtain a version of \ref{t:mainthm} (ii) with a scalar-valued density $\tilde q_t$ instead of the vector-valued density $q_t$. Namely one can let $\tilde q_t$ cycle through the different components of $q_t$, for example by setting $\tilde q_t=Ke_1\cdot q_{Kt}$ on the interval $\left[0,\frac{1}{K}\right)$, $\tilde q_t=Ke_2\cdot q_{Kt-1}$ on the interval $\left[\frac{1}{K},\frac{2}{K}\right)$, and continuing cyclically in this manner. In fact, precisely this argument is used in the proof of Theorem \ref{t:fields_in_Falpha} below. 
In principle, a similar trick would also work if $\HH$ is infinite-dimensional with countable orthonormal basis. Then one would need to cycle through the components of $\Q$ with increasing speed (say, by spending time $\frac{1}{2^j}$ on the $j$-th component $e_j\cdot q_t$). However, this comes at the cost of deteriorated estimates on the decay of $\tilde q_t$. In our cases of interest, $\HH$ is finite-dimensional, and so we do not elaborate more on this.

Finally, let us point out that Muirhead's construction has the property that the resulting $\q$ is pointwise non-negative. This implies in particular that the Gaussian fields with covariance $\q(\cdot,t)*_1\q(\cdot,t)$ have non-negative correlations and hence satisfy the FKG property. With our construction (like with Bauerschmidt's) it seems that one cannot recover this non-negativity property. Namely even if $w_t$ is pointwise non-negative, this does not imply that $w_t(L)$ is pointwise non-negative (unless $L$ maps non-negative functions to non-negative functions, which is not true in the examples of interest).
Let us remark, though, that in the setting of Theorem \ref{t:mainthm} (i), the operators $Q_t$ themselves (and not just $Q_t^*Q_t$) happen to be self-adjoint and positive semidefinite, as follows from the fact that the functions $\tilde w_t$ in Lemma \ref{l:existence_wtilde} are non-negative. In the setting of Theorem \ref{t:mainthm} (ii) however, while one might consider the components of $Q_t$ with respect to a basis of $\HH$ as operators on $L^2(X)$, they are in general not even self-adjoint.

\paragraph{Possibility of an easier construction}

The reader might wonder whether any positive quadratic form $\Phi(u,v)$ with finite range $r$ on $L^2(X)$ can be written as $(Qu,Qv)_{L^2(X,\HH)}$ with some $\HH$ and some $Q\colon L^2(X)\to L^2(X,\HH)$ of range $r'$, where $\HH$ may depend on $X$ and $r$, but not on $\Phi$, and where $r'$ is as small as possible.

If this were possible with some fixed $\HH$ and $r'\le Cr$, then it would be straightforward to obtain a finite-range decomposition as in \eqref{e:finite_range_new2} from the ones like in \eqref{e:finite_range}, and so Theorem \ref{t:mainthm} would not be interesting.

If $d=1$, then this is trivially possible. Indeed, the Cholesky decomposition of a band matrix is banded with the same bandwidth, and so one can even take $\HH=\R$ and $r'=r$.

For $d\ge2$ there are positive results in some special cases: In \cite[Corollary 3.2]{EGR04} it is shown that if $X=\R^d$ and the kernel $\phi$ of $\Phi$ is radial and pointwise non-negative, then it can be written as $q*q$ for some radial kernel $q$, and the associated operator $Q$ then has the desired properties (with $\HH=\R$, even). Moreover, it follows from \cite[display (3)]{R70} that even without the assumption of pointwise non-negativity of $\phi$, one can write $\phi=\sum_k q_k*q_k$ for a finite or countable sequence $(q_k)$. This leads to a representation of $Q$ with values in $\HH=l^2(\N)$, say.

In general, however, we have a strong indication that for $d\ge2$ it is not possible to find $Q$ with the desired properties. Namely take $X=\Z^d$, and assume that $\Phi$ is translation-invariant. Then it is reasonable to assume that $Q$ should be translation-invariant as well, and that $\HH=\R^K$ for some $K$. Now a translation-invariant bilinear form is given by a multidimensional Toeplitz-matrix. The symbol of this matrix is then a polynomial in $d$ complex variables of degree $\le r$ in each variable, and being positive definite is equivalent to this symbol being positive on the complex torus $\{|z_1|=\ldots=|z_d|=1\}$, i.e. as a $d$-dimensional trigonometric polynomial. Existence of $Q$, on the other hand, reduces to writing this symbol as a sum of squares of $K$ trigonometric polynomials of degree at most $r'$ in each variable.

Now if $r'=r$, then this task is not possible in general for any $K$, as shown in \cite{CP52} and independently in \cite{R63}. This is closely related to Hilbert's 17th problem and in particular to the fact that there are polynomials in $\ge2$ variables that are non-negative, yet cannot be written as the sum of squares of polynomials.

If $r'>r$, there is more freedom to choose the polynomials, though. This topic is discussed in \cite{D04,GL06}. In those works it is shown that a general strictly positive trigonometric polynomial can be written as the sum of squares of trigonometric polynomials. However, already in case $d=2$ both the number of polynomials and their degree grow like $r^2$ (see \cite[Theorem 2.1]{GL06}). So in our setting one would need to take the dimension of $\HH$ and the range $r'$ to be both of order at least $r^2$. This is clearly not useful for our purposes.

There seems to be no proof that this quadratic growth is optimal. However, in light of these results, it seems unlikely that one can achieve $r'\le Cr$ for some fixed $C$, and so we believe that it is highly unlikely that there is a trivial way to obtain a finite-range decomposition as in \eqref{e:finite_range_new2} from the ones like in \eqref{e:finite_range}.

\section{Proofs}

\subsection{Existence of a white noise decomposition under \texorpdfstring{\eqref{e:P}}{(P)}}\label{s:exP}

As mentioned in the introduction, the proof of Theorem \ref{t:mainthm} under assumption \eqref{e:P} and under assumptions \eqref{e:Pstar} and \eqref{e:F} is quite different. The former requires just minor modifications to the corresponding argument in \cite{B13}, and so we begin with it.

The main step in the proof of Theorem \ref{t:mainthm} under assumption \eqref{e:P} is the verification of the following lemma, the analogue of \cite[Lemma 2.1]{B13}.

\begin{lemma}\label{l:existence_wtilde}
Assume that the regular closed symmetric positive semidefinite quadratic form $E$ and its generator $L$ satisfies \eqref{e:P}. Then there is a family of non-negative functions $\tilde w_t\in C^\infty(\R)$ that depend smoothly on $t$ and have the following properties:

We have that
\begin{equation}\label{e:existence_wtilde1}
\frac{1}{\lambda}=\int_0^\infty t^{(2-\gamma)/\gamma} |\tilde w_t(\lambda)|^2\ud t
\end{equation}
for all $\lambda\in \spec(L)\setminus\{0\}$.

Moreover, we have the bound
\begin{equation}\label{e:existence_wtilde2}(1+t^2\lambda)^l\lambda^m\left|\frac{\partial^m}{\partial\lambda^m}\tilde w_t(\lambda)\right|\le C_{l,m}
\end{equation}
for any $l,m\in\N$, any $\lambda\in \spec(L)\setminus\{0\}$, and any $t>0$.

Finally, we have that
\begin{equation}\label{e:existence_wtilde3}
\supp (\tilde w_t(L)u)\subset N_{\theta(t)}(\supp(u))
\end{equation}
for any $u\in C_c(X)$.
\end{lemma}

In the proof we will use the Fourier transform which we define as
\[\hat u(\xi)=\frac{1}{2\pi}\int_R u(x)\e^{-i\xi x}\ud x\]
for $u\in C_c(\R)$, say.

\begin{proof}
The argument is very similar to the one in \cite[Lemma 2.1]{B13}. The only difference is that we pick $\varphi$ below so that \eqref{e:existence_wtilde4} holds with $\varphi^2$ on the right-hand side instead of $\varphi$ as in \cite[Equation (2.22)]{B13}.
Namely we pick a function $\kappa$ whose Fourier transform $\hat\kappa$ is a smooth non-negative real-valued function with support in $\left[-\frac12,\frac12\right]$, and let $\varphi=|\kappa|^2$. Then $\varphi\in C^\infty(\R)$, and $\supp(\hat\varphi)\subset[-1,1]$. In addition,
\[\lambda\mapsto\int_0^\infty t^{(2-\gamma)/\gamma}(\varphi(\lambda^{\gamma/2}t))^2\ud t\]
is a non-negative homogenous function of degree $-1$, and so there is a constant $c_0>0$ such that
\begin{equation}\label{e:existence_wtilde4}
\frac{1}{\lambda}=c_0\int_0^\infty t^{(2-\gamma)/\gamma}(\varphi(\lambda^{\gamma/2}t))^2\ud t
\end{equation}
for all $t>0$. We now define 
\[\tilde w_t(\lambda)=\sqrt{c_0}\varphi(\lambda^{\gamma/2}t).\]
With this definition, \eqref{e:existence_wtilde1} is obvious. The estimate \eqref{e:existence_wtilde2} follows from the chain rule and the fact that $\varphi$ is a Schwartz function. Finally, the finite-range property follows as in \cite{B13} from the fact that
\[\tilde w_t(L)u=\sqrt{c_0}\int_{-1}^{1}\hat\varphi(\xi)\cos(L^{\gamma/2}st)u\ud s\]
and assumption \eqref{e:P}.
\end{proof}
Using Lemma \ref{l:existence_wtilde}, the proof of our main theorem under the assumption \eqref{e:P} is now easy.

\begin{proof}[Proof of Theorem \ref{t:mainthm} (i)]
We define $Q_t=t^{(2-\gamma)/(2\gamma)}\tilde w_t(L)$. With this choice it follows immediately from Lemma \ref{l:existence_wtilde} that $Q_t$ has range at most $\theta(t)$, that we have \eqref{e:finite_range_mainthm} and that the operator norm of $Q_t$ is bounded by $C_\gamma t^{(2-\gamma)/(2\gamma)}$.

The fact that under assumption \eqref{e:H} $Q_t$ is given by convolution with a continuous kernel $q_t$ is non-trivial, but follows from exactly the same argument as in \cite{B13}. Following that argument, we also find the bounds
\[\|Q_t\delta_x\|_{L^2(X)}\le t^{(2-\gamma)/(2\gamma)}\frac{C_{\alpha,\gamma}\sqrt{\omega(x)}}{t^{\alpha/(2\gamma)}}=\frac{C_{\alpha,\gamma}\sqrt{\omega(x)}}{t^{(\alpha+\gamma-2)/(2\gamma)}}\]
and
\begin{equation}\label{e:mainthm0}
(\delta_x,Q_t\delta_x)_{L^2(X)}\le C_{\alpha,\gamma}t^{(2-\gamma)/(2\gamma)}\frac{\sqrt{\omega(x)}}{t^{\alpha/(2\gamma)}}\frac{\sqrt{\omega(y)}}{t^{\alpha/(2\gamma)}}=\frac{C_{\alpha,\gamma}\sqrt{\omega(x)\omega(y)}}{t^{(2\alpha+\gamma-2)/(2\gamma)}}
\end{equation}
for any $x,y\in X$. The former estimate implies \eqref{e:estimate_kernel2}, the latter \eqref{e:estimate_kernel}.
\end{proof}

\subsection{Existence of a white noise decomposition under \texorpdfstring{\eqref{e:Pstar} and \eqref{e:F}}{(P*) and (F)}}\label{s:exPstar}
The proof of the main theorem under assumptions \eqref{e:Pstar} and \eqref{e:F} is less easy and in particular cannot be directly adapted from \cite{B13}. Let us begin by quickly explaining why that is the case.

As explained in the introduction, Bauerschmidt's construction relies on defining a family of functions $w_t$ such that \eqref{e:identity_lambda} holds and the operators $w_t(L)$ have the desired finite range. If we consider the easiest case $\gamma=1$, then under assumption \eqref{e:Pstar} it turns out that one can choose
\[w_t(\lambda)=c_0\sum_{n\in\Z}\varphi^2\left(\arccos\left(1-\frac{\lambda}{B}\right)t-2\pi n t\right)\]
as this turns out to be a polynomial in $\lambda$ (so that \eqref{e:Pstar} allows to control the range of $w_t(L)$). If we wanted to use the same approach as in Section \ref{s:exP} we would instead need to ensure that $w_t=|\tilde w_t(\lambda)|^2$ for a polynomial $\tilde w_t(\lambda)$ that satisfies the analogues of \eqref{e:existence_wtilde1} and \eqref{e:existence_wtilde2}. Looking at the proof of Lemma \ref{l:existence_wtilde}, an obvious ansatz would be to choose
\[\tilde w_t(\lambda)=C\sum_{n\in\Z}\varphi\left(\arccos\left(1-\frac{\lambda}{B}\right)t-2\pi nt\right)\]
for a suitable constant $C$. With this definition \eqref{e:existence_wtilde2} and \eqref{e:existence_wtilde3} would follow from exactly the same argument as in \cite[Lemma 2.3]{B13}. However, \eqref{e:existence_wtilde1} does not follow from \eqref{e:existence_wtilde4}. The issue is that when computing $|\tilde w_t(\lambda)|^2$ one encounters cross-terms like \[\varphi\left(\arccos\left(1-\frac{\lambda}{B}\right)t-2\pi nt\right)\varphi\left(\arccos\left(1-\frac{\lambda}{B}\right)t-2\pi n't\right)\] for $n\neq n'$ that do not have a clear scaling in $t$, and so 
we can no longer use \eqref{e:existence_wtilde4} to compute $\int_0^\infty t|\tilde w_t(\lambda)|^2\ud t$. While this might sound like a mere technicality that should be easy to overcome, the underlying issue is that the periodization in the definitions of $w_t$ and $\tilde w_t$ does not commute with taking a square. It is not clear how to overcome this obstacle, and we have not been able to find a variant of Bauerschmidt's construction where $w_t$ is the square of a polynomial.

Instead we follow the approach sketched in the introduction, and try to write $w_t$ as the sum of several polynomials (and one additional factor $(2B)^\gamma-\lambda^\gamma$). The main new tool that allows us to do this will be the following lemma by P\'olya and Szeg\H o.
\begin{lemma}\label{l:polyaszego}
If $s$ is a polynomial that is non-negative for non-negative $x$, then there are polynomials $a^{(j)}$ for $j\in\{1,2,3,4\}$ with real coefficients such that
$\deg a^{(j)}\le \deg s$ for $j\in\{1,2\}$, $\deg a^{(j)}\le \deg s-1$ for $j\in\{3,4\}$ and such that
 \[
s(x)=(a^{(1)}(x))^2+(a^{(2)}(x))^2+x\left((a^{(3)}(x))^2+(a^{(4)}(x))^2\right).\]
\end{lemma}
\begin{proof}
This is the subject of \cite[Chapter VI, Problem 45]{PS98}. The bound on the degree of the $a_t^{(j)}$ is not spelled out there, but one can obtain it by following the proof of the result sketched there.
\end{proof}
We would like to apply Lemma \ref{l:polyaszego} to a family of polynomials $s_t$ parametrized by $t\in\R^+$ that depend continuously on $t$. In that setting a natural question is whether one can choose the $a^{(j)}_t$ to also depend continuously on $t$ \footnote{Note that for the purpose of constructing a white noise decomposition as in Definition \ref{d:Falpha}, we do not really need continuous dependence on $t$, but only measurable dependence. But even the latter statement would need an argument, and so we prove the stronger statement right away.}. The following lemma answers this positively under some small additional conditions. We let $C^0_{\text{loc}}(\R)$ be the space of continuous functions on $\R$ equipped with the topology of locally uniform convergence, and note that a sequence of polynomials of bounded degree converges in $C^0_{\text{loc}}(\R)$ if and only if all coefficients converge.
\begin{lemma}\label{l:polyaszego_cont}
Let $I\subset\R$ be an open interval, and let $s_t$ for $t\in I$ be a family of real-valued polynomials that are non-negative for non-negative $x$. Suppose that $s_t(0)>0$ for all $t\in I$, that $\deg s_t$ is locally bounded, and such that $s\colon I\to C^0_{\text{loc}}(\R)$ is continuous. Then there are polynomials $a^{(j)}_t$ for $t\in I$ and $j\in\{1,2,3,4\}$ with real coefficients such that $a^{(j)}\colon I\to C^0_{\text{loc}}(\R)$ is continuous for $j\in\{1,2,3,4\}$,
$\deg a^{(j)}_t\le \deg s_t$ for $j\in\{1,2\}$, $\deg a^{(j)}_t\le \deg s_t-1$ for $j\in\{3,4\}$ and such that
 \[
s_t(x)=(a^{(1)}_t(x))^2+(a^{(2)}_t(x))^2+x\left((a^{(3)}_t(x))^2+(a^{(4)}_t(x))^2\right).\]
\end{lemma}
Here it might be possible to remove the assumption that $s_t(0)>0$ for all $t\in I$. However, this assumption simplifies the proof and is easily checked in our application, and so we kept it.

The proof of Lemma \ref{l:polyaszego} in \cite[Chapter VI, Problem 45]{PS98} describes an algorithm how to construct the $a_j$ given the factorization of $s$ into irreducible factors over $\R$. We will argue that the steps of this algorithm can be executed in such a way that the resulting polynomials depend continuously on the coefficients of $s$. Actually it turns out that the algorithm as described in \cite[Chapter VI, Problem 45]{PS98} does not achieve this, but by making a minor change in the way quadratic irreducible factors are dealt with we can fix this.
\begin{proof}$ $
\emph{Step 1: Factorization into irreducible factors}\\
We claim that we can write
\begin{equation}\label{e:polyaszego_cont1}
s_t(x)=s_t(0)\prod_{j=1}^\infty\left(1-\frac{x}{z_{t,j}}\right)
\end{equation}
where the $z_{t,j}\colon\R\cup\{+\infty\}$ are continuous and for each $t$ all but finitely many of the $z_{t,j}$ are equal to $+\infty$. Indeed, for each fixed $t$,  if $\tilde z_j$, $j\in\{1,2,\ldots,\deg s_t\}$ are the complex zeroes of $s_t$, then none of them can be 0, and so we can write
\[s_t(x)=C\prod_{j=1}^{\deg s_t}\left(1-\frac{x}{\tilde z_j}\right).\]
Comparing coefficients, we see that the constant $C$ has to be $s_t(0)$, and so we obtain a representation like \eqref{e:polyaszego_cont1}. In view of this, it only remains to check that one can number the zeroes of $s_t$ in such a way that $t\mapsto z_{t,j}$ is continuous for each $t$. If the degree of $s_t$ is constant, this is classical (cf. for example \cite{AKML98}). In our setting the degree is not necessarily constant and so we have zeroes coming in from infinity, but the classical proof still applies.

\emph{Step 2: Simplified version}\\
As a next step, we claim that we can find polynomials $b_t^{(1)}$ and $b_t^{(2)}$ that only take non-negative values, depend continuously on $t$ and are such that 
\begin{equation}\label{e:polyaszego_cont2}
s_t(x)=s_t(0)\left(b_t^{(1)}(x)+xb_t^{(2)}(x)\right).
\end{equation}

The key observation for that purpose is that if we can write two polynomials $s$ and $\tilde s$ as $b^{(1)}(x)+xb^{(2)}(x)$ and $\tilde b^{(1)}(x)+x\tilde b^{(2)}(x)$, respectively, then their product is of the same form. Indeed, we have
\begin{equation}\label{e:polyaszego_cont3}s(x)\tilde s(x)=\left(b^{(1)}(x)\tilde b^{(1)}(x)+x^2b^{(2)}(x)\tilde b^{(2)}(x)\right)+x\left(b^{(1)}(x)\tilde b^{(2)}(x)+b^{(2)}(x)\tilde b^{(1)}(x)\right).
\end{equation}
Even more, the binary operation sending 
\begin{align*}
&\left(\left(b^{(1)}(x),b^{(2)}(x)\right),\left( \tilde b^{(1)}(x),\tilde b^{(2)}(x)\right)\right)\\
&\qquad\mapsto\left(b^{(1)}(x)\tilde b^{(1)}(x)+x^2b^{(2)}(x)\tilde b^{(2)}(x),b^{(1)}(x)\tilde b^{(2)}(x)+b^{(2)}(x)\tilde b^{(1)}(x)\right)
\end{align*}
is commutative and associative, as a straightforward calculation shows. So it suffices to write the individual factors in \eqref{e:polyaszego_cont1} in the form $b^{(1)}(x)+xb^{(2)}(x)$.

For this purpose, we use that $s_t$ is real-valued and takes non-negative values for non-negative $x$. This means that the $z_{t,j}$ that are finite are either real and negative, or complex in which case their complex conjugate is also a zero (this includes the case of positive real zeroes with have to have even multiplicity, i.e. appear in pairs as well). Zeroes of the former type are easy to deal with, as if $z_{t,j}$ is real and negative, then $1-\frac{x}{z_{t,j}}$ is already in the required form. Zeroes of the latter type can be grouped with their complex conjugate, and we rewrite them as
\begin{equation}\label{e:polyaszego_cont4}
\left(1-\frac{x}{z_{t,j}}\right)\left(1-\frac{x}{\overline{ z_{t,j}}}\right)=\left(\frac{(x-\Re z_{t,j}\vee0)^2+(\Re z_{t,j}\wedge0)^2+(\Im z_{t,j})^2}{|z_{t,j}|^2}\right)+x\left(-\frac{2\Re z_{t,j}\wedge0}{|z_{t,j}|^2}\right).
\end{equation}
Clearly both polynomials in brackets are non-negative.
Let us remark that in  \cite{PS98}, the left-hand side is instead rewritten as $\frac{(x-\Re z_{t,j})^2+(\Im z_{t,j})^2}{|z_{t,j}|^2}$. This expression is arguably easier, but using it continuity in $t$ would fail (as we will see shortly).

We now consider the decomposition \eqref{e:polyaszego_cont1}, group the pairs of complex linear factors and rewrite them according to \eqref{e:polyaszego_cont4}, and then use the relation \eqref{e:polyaszego_cont3} to arrive at the representation \eqref{e:polyaszego_cont2}. Because of the commutativity and associativity, the only way that continuity could be lost is when two real zeroes collide and leave the real axis (i.e. when we go from two linear factors to one irreducible quadratic factor). However, this is not a problem because of the specific choice we made in \eqref{e:polyaszego_cont4}. Namely such a collision can only happen if $\Re z_{t,j}<0$, and assuming that and taking the limit $\Im z_{t,j}\to 0$, the right-hand side in \eqref{e:polyaszego_cont4} becomes equal to 
\[\left(\frac{x^2+(\Re z_{t,j})^2}{|z_{t,j}|^2}\right)+x\left(-\frac{2\Re z_{t,j}}{|z_{t,j}|^2}\right)\]
This is the same expression as we obtain when we apply \eqref{e:polyaszego_cont3} to $\left(1-\frac{x}{z_{t,j}}\right)\left(1-\frac{x}{z_{t,j}}\right)$. So, no matter whether a pair of zeroes approaches the same point from the real or imaginary direction, the resulting polynomials stay continuous.

\emph{Step 3: Full version}\\
We have shown that there is a decomposition \eqref{e:polyaszego_cont2} with the $b_t^{(j)}$ non-negative and depending continuously on $t$. It remains to show that we can write each $b_t^{(j)}$ as the sum of squares of two polynomials that depend continuously on $t$. The argument for this is analogous to the one we used to arrive at \eqref{e:polyaszego_cont2}, so let us only sketch the most important steps. First of all, as $s_t(0)>0$, we must have $b_t^{(1)}(0)>0$. For $b_t^{(2)}$, following the argument in Step 2 and arguing inductively, we see that either $b_t^{(2)}$ is identically 0 (in which case it is trivially the sum of squares of two polynomials) or we must have $b_t^{(2)}(0)>0$.
So we can assume that $b_t^{(j)}(0)>0$ for $j\in\{1,2\}$. We can now find a factorization as in \eqref{e:polyaszego_cont1}. Now as $b_t^{(j)}$ is non-negative, its real zeroes must come in pairs (which yields a trivial representation as a square of a polynomial plus $0^2$), while the complex zeroes must have another zero which is conjugate and can be written as
\[
\left(1-\frac{x}{z_{t,j}}\right)\left(1-\frac{x}{\overline{ z_{t,j}}}\right)=\left(\frac{(x-\Re z_{t,j})^2}{|z_{t,j}|^2}\right)+\left(\frac{(\Im z_{t,j})^2}{|z_{t,j}|^2}\right).
\]
Finally, the analogue of \eqref{e:polyaszego_cont2} is that if $s(x)=(a^{(1)}(x))^2+(a^{(2)}(x))^2$ and $\tilde s(x)=(\tilde a^{(1)}(x))^2+(\tilde a^{(2)}(x))^2$, respectively, then
\[
s(x)\tilde s(x)=\left(a^{(1)}(x)\tilde a^{(1)}(x)-a^{(2)}(x)\tilde a^{(2)}(x)\right)^2+\left(a^{(1)}(x)\tilde a^{(2)}(x)+a^{(2)}(x)\tilde a^{(1)}(x)\right)^2.
\]
The corresponding binary operation is again commutative and associative, and we can now argue as in Step 2 to complete the proof.
\end{proof}
\begin{remark}
In our application the family $s_t$ will depend not just continuously, but smoothly on $t$. In view of this, it is natural to wonder whether this allows to define the $a^{(j)}_t$ in Lemma \ref{l:polyaszego_cont} so that they depend smoothly on $t$ as well. However, the answer to that is no. This is related to the fact that the zeroes of a family of polynomials in general do not depend smoothly on the coefficients, cf. \cite{AKML98}.

Consider for example the family $s_t(x)=x^2+f(t)$ for a function $f\in C^\infty(\R)$ with $f\ge0$. Then we must have $a^{(1)}_t(x)=x$ and $a^{(2)}_t(x)=\pm\sqrt{f(t)}$, or vice versa. However, there exist non-negative smooth functions $f$ whose square root is not twice  differentiable. So, while it might be possible to choose the $a^{(j)}_t$ to be once differentiable in $t$, smoothness in $t$ is impossible.
\end{remark}

We can now state and prove an analogue of Lemma \ref{l:existence_wtilde}. Just like that lemma was the main step in the first part of the proof of Theorem \ref{t:mainthm}, Lemma \ref{l:existence_aj} will be the main step in the proof of the second part of the theorem.

\begin{lemma}\label{l:existence_aj}
Assume that the regular closed symmetric positive semidefinite quadratic form $E$ and its generator $L$ satisfy \eqref{e:Pstar}.
Then there are families of functions $a^{(j)}_t\in C^\infty(\R)$ that depend continously on $t$ and have the following properties:

We have that 
\begin{equation}\label{e:existence_aj1}
\frac{1}{\lambda}=\int_0^\infty t^{(2-\gamma)/\gamma}w_t(\lambda)\ud t
\end{equation}
for all $\lambda\in [0,B]$, where
\begin{equation}\label{e:existence_aj2}
w_t(\lambda):=(a_t^{(1)}(\lambda))^2+(a_t^{(2)}(\lambda))^2+((2B)^\gamma-\lambda^\gamma)\left((a_t^{(3)}(\lambda))^2+(a_t^{(4)}(\lambda))^2\right).
\end{equation}

Moreover, we have the bounds
\begin{align}
(1+t^{2/\gamma}\lambda)^l\lambda^m\left|\frac{\partial^m}{\partial\lambda^m}w_t(\lambda)\right|&\le C_{B,l,m}\label{e:existence_aj3},\\
(1+t^{2/\gamma}\lambda)^l\left|a^{(j)}_t(\lambda)\right|&\le C_{B,l}\label{e:existence_aj4}
\end{align}
for any $l,m\in\N$, any $\lambda\in [0,B]$, any $j\in\{1,2,3,4\}$ and any $t\ge1$. For $t<1$ we have instead the explicit formulas $w_t(\lambda)=\frac{\tilde F_{\gamma,B}(t)}{t}$, $a^{(1)}_t=\sqrt{w_t(\lambda)}$, $a^{(2)}_t=a^{(3)}_t=a^{(4)}_t=0$ for some bounded function $\tilde F_{\gamma,B}\colon[0,1]\to[0,\infty)$.

Finally, we have that
\begin{align}
\supp (a^{(j)}_t(L)u)&\subset N_{\theta(t)}(\supp(u))\quad\forall j\in\{1,2\},\label{e:existence_aj6}\\
\supp (a^{(j)}_t(L)u)&\subset N_{\theta(t-1)}(\supp(u))\quad\forall j\in\{3,4\}\label{e:existence_aj7}
\end{align}
for any $u\in C_c(X)$ and any $t$.
\end{lemma}

\begin{remark}\label{r:markov}
It is also possible to derive estimates for $\left|\frac{\partial^m}{\partial\lambda^m}a^{(j)}_t(\lambda)\right|$ for $m\ge1$, for example by using the fact that $a^{(j)}_t$ is a polynomial in $\lambda^\gamma$ of degree $\le t$ together with the Markov brothers' inequality \cite{MG16}. However, the estimates that one obtains in this manner are far from optimal scalingwise, and so we do not work out the details here.
\end{remark}

\begin{proof}$ $

\emph{Step 1: Construction of $\bar w_t$}\\
We begin by constructing a function $\bar w_t$ that has all the properties that we require of $w_t$, except possibly for $t<1$. In the next step we will modify $\bar w_t$ slightly to obtain the actual $w_t$.

In case $\gamma=1$ this can be done just like in \cite{B13}\footnote{There is a small mistake in the published version of \cite{B13} that is corrected in the most recent arXiv version: If working under \eqref{e:Pstar} one needs to treat the case $t<1$ separately as various estimates cease to hold there. We quote the results from \cite{B13} in the corrected form.}. In the general case $\frac{1}{\gamma}\in\N$ we can still follow the same approach. The main difference is that we need to use the partial fraction decomposition of $\frac{1}{(1-\cos z)^{1/\gamma}}$. For this, \cite{ST12} is a convenient reference. Indeed, according to \cite[Theorem 3]{ST12}, there are numbers $a_j\in\Q$ for $j\in\left\{0,1,\ldots,\frac{1}{\gamma}-1\right\}$ such that we have the identity
\begin{equation}\label{e:existence_aj8}\left(\frac{1}{1-\cos x}\right)^{1/\gamma}=\left(\frac{1}{2\sin^2(x/2)}\right)^{1/\gamma}=\sum_{n\in\Z}\sum_{j=0}^{1/\gamma-1}\frac{a_j}{(x-2\pi n)^{2/\gamma-2j}}.
\end{equation}
There are explicit formulas for the $a_j$ in terms of Bernoulli numbers. We, however, only need the fact that the $a_j$ are all non-negative. This follows from the recursive formula in \cite[Equation (3)]{ST12} together with a straightforward induction.

Next let $\varphi$ be as in Lemma \ref{l:existence_wtilde}, and recall that $\varphi^2$ is a non-negative function whose Fourier transform is non-negative and has compact support in $(-1,1)$. The same argument that led to \eqref{e:existence_wtilde4} also shows that there are constants $c_k'$ such that
\begin{equation}\label{e:existence_aj9}
\frac{1}{\lambda^{k+1}}=c_k'\int_0^\infty t^{2k+1}\varphi^2(\lambda^{1/2}t)\ud t
\end{equation}
for all $\lambda>0$ and all $k\in\N$.
We now consider the function
\begin{equation}\label{e:existence_aj10}
\bar w_t(\lambda)=\frac{1}{2B}\sum_{n\in\Z}\sum_{j=0}^{1/\gamma-1}c_{1/\gamma-j-1}'\frac{a_j}{t^{2j}}\varphi^2\left(\arccos\left(1-\left(\frac{\lambda}{2B}\right)^\gamma\right)t-2\pi n t\right)
\end{equation}
for $\lambda\in[0,2B]$, say.

Let us prove that this function has the required properties. First of all, using \eqref{e:existence_aj9} with $\lambda$ replaced by $\left(\arccos\left(1-\left(\frac{\lambda}{2B}\right)^\gamma\right)-2\pi n\right)^2$ and then \eqref{e:existence_aj8} we see that
\begin{equation}\label{e:existence_aj11}
\begin{split}
&\int_0^\infty t^{(2-\gamma)/\gamma}\bar w_t(\gamma)\ud t\\
&\quad=\frac{1}{2B}\sum_{j=0}^{1/\gamma-1}c_{1/\gamma-j-1}'a_j\sum_{n\in\Z}\int_0^\infty t^{2/\gamma-2j-1}\varphi^2\left(\arccos\left(1-\left(\frac{\lambda}{2B}\right)^\gamma\right)t-2\pi n t\right)\ud t\\
&\quad=\frac{1}{2B}\sum_{j=0}^{1/\gamma-1}a_j\sum_{n\in\Z}\frac{1}{\left(\arccos\left(1-\left(\frac{\lambda}{2B}\right)^\gamma\right)-2\pi n\right)^{2/\gamma-2j}}\\
&\quad=\frac{1}{2B}\left(\frac{1}{\left(\frac{\lambda}{2B}\right)^\gamma}\right)^{1/\gamma}\\
&\quad=\frac{1}{\lambda}.
\end{split}
\end{equation}

Next, as in \cite{B13} the Poisson summation formula allows us to rewrite
\begin{equation}\label{e:existence_aj12}
\begin{split}
\bar w_t(\lambda)&=\frac{1}{2B}\sum_{j=0}^{1/\gamma-1}c_{1/\gamma-j-1}'\frac{a_j}{t^{2j}}\sum_{k\in\Z}\frac{1}{t}\widehat{\varphi^2}\left(\frac{k}{t}\right)\cos\left(k\arccos\left(1-\left(\frac{\lambda}{2B}\right)^{\gamma}\right)\right)\\
&=\frac{1}{2B}\sum_{j=0}^{1/\gamma-1}c_{1/\gamma-j-1}'\frac{a_j}{t^{2j+1}}\sum_{k\in\Z}\widehat{\varphi^2}\left(\frac{k}{t}\right)T_k\left(1-\left(\frac{\lambda}{2B}\right)^{\gamma}\right)
\end{split}
\end{equation}
where $T_k(x):=\cos(k\arccos x)$ is the $k$-th Chebyshev polynomial of the first kind. As $\widehat{\varphi}$ is supported in $(-1,1)$, the summands in \eqref{e:existence_aj9} vanish except $|k|\le t$. Thus $\bar w_t$ is the restriction to $[0,2B]$ of a polynomial in $\lambda^\gamma$ of degree at most $t$.

\emph{Step 2: Construction of $w_t$}\\
From \eqref{e:existence_aj12} we can read off that for $t\le1$ we have
\[\bar w_t(\lambda)=\frac{1}{2B}\sum_{j=0}^{1/\gamma-1}c_{1/\gamma-j-1}'\frac{a_j}{t^{2j+1}}\widehat{\varphi^2}(0).\]
This means that $\bar w_t$ does not behave like $\frac{C}{t}$ for $t$ small, as we were hoping\footnote{Note that for applications the main interest is in the decay rate of $Q_t$ for $t$ large, and so it would not matter much to have non-optimal asymptotics of $w_t$ for $t$ small. Nonetheless, we explain here how to improve the asymptotics for $t$ small, as the argument is very short.}. An easy way to fix this while preserving the other properties of $\bar w_t$ (in particular the continuous dependence on $t$) is as follows. Fix some continuous function $\iota_\gamma\colon[0,1]\to[0,\infty)$ such that $\int_0^1t^{(2-2\gamma)/\gamma}\iota_\gamma(t)\ud t=1$, $\iota_\gamma(1)=0$ and let
\[\Gamma_{\gamma,B}:=\int_0^1t^{(2-\gamma)/\gamma}\left(\bar w_t(\lambda)-\frac{\bar w_1(\lambda)}{t}\right)\ud t\]
for some $\lambda$. Note that $\Gamma_{\gamma,B}$ does not depend on the choice of $\lambda$ (as $\bar w_t$ is constant for $t\le1$), and that $\Gamma_{\gamma,B}\ge0$ (as $t\bar w_t(\lambda)$ is decreasing in $t$).
Then define
\[w_t(\lambda)=\begin{cases}\frac{1}{t}\left(\bar w_1(\lambda)+\Gamma_{\gamma,B}\iota_\gamma(t)\right)&t<1\\\bar w_t(\lambda)& t\ge1\end{cases}.\]
With this definition clearly $w_t\in C^\infty(\R)$ for each $t$, and $w_t$ depends continuously on $t$. Moreover, we have
\begin{align*}
\int_0^1t^{(2-\gamma)/\gamma}w_t(\lambda)\ud t&=\int_0^1t^{(2-2\gamma)/\gamma}(\bar w_1(t)+\Gamma_{\gamma,B}\iota_\gamma(t))\ud t\\
&=\int_0^1t^{(2-2\gamma)/\gamma}\bar w_1(\lambda)\ud t+\Gamma_{\gamma,B}\\\
&=\int_0^1t^{(2-\gamma)/\gamma}\bar w_t(\lambda)\ud t
\end{align*}
and so \eqref{e:existence_aj11} directly implies \eqref{e:existence_aj1}. Moreover, the bound \eqref{e:existence_aj3} for $t\ge1$ (where $w_t=\bar w_t$) follows from exactly the same argument as in \cite{B13}, and the explicit formula for $t<1$ can be read off directly.

\emph{Step 3: Construction of the $a_t^{(j)}$}\\
In order to construct the $a_t^{(j)}$, we will employ Lemma \ref{l:polyaszego_cont}. Let $v_t(\mu):=w_t(\mu^{1/\gamma})$, so that $v_t$ is the restriction to $[0,(2B)^\gamma]$ of a polynomial. Slightly abusing notation, we denote this polynomial by $v_t$ as well, i.e. we define $v_t$ as a polynomial on $\R$ by
\begin{equation}\label{e:existence_aj13}
v_t(\mu)=\frac{1}{2B}\sum_{j=0}^{1/\gamma-1}c_{1/\gamma-j-1}'\frac{a_j}{t^{2j+1}}\sum_{k\in\Z}\widehat{\varphi^2}\left(\frac{k}{t}\right)T_k\left(1-\frac{\mu}{(2B)^\gamma}\right).
\end{equation}

We claim that $v_t(\mu)$ is non-negative for $\mu\in(-\infty,(2B)^\gamma]$ and that $v_t((2B)^\gamma)>0$. Indeed, the formula \eqref{e:existence_aj10} implies that $w_t(\lambda)$ is non-negative for $\lambda\in[0,2B]$, and thus $v_t(\mu)$ is non-negative for $\mu\in[0,(2B)^\gamma]$. On the other hand, for $\mu<0$ we have in particular $1-\frac{\mu}{(2B)^\gamma}>1$. Each Chebyshev polynomial $T_k(x)$ is non-negative for $x\ge1$, and the coefficients in \eqref{e:existence_aj12} are all non-negative by our assumptions on $\varphi$. Thus the formula \eqref{e:existence_aj9} shows that $v_t(\mu)$ is non-negative for $\mu<0$ as well. Finally, using \eqref{e:existence_aj12} again, we see that
\[v_t((2B)^\gamma)\ge \frac{c_{1/\gamma-1}'a_0}{t}\widehat{\varphi^2}(0)T_0(0)>0\]
We have seen that $v_t(\mu)$ is non-negative for $\mu\le (2B)^\gamma$, strictly positive for $\lambda=(2B)^\gamma$ and moreover its coefficients depend continuously on $t$. Applying now Lemma \ref{l:polyaszego_cont} to the family of polynomials $v_t((2B)^\gamma-\cdot)$, we find that there exist polynomials $b_t^{(j)}$ for $j\in\{1,2,3,4\}$ of degree $\le t$ for $j\in\{1,2\}$ and of degree $\le t-1$ for $j\in\{3,4\}$ that depend continuously on $t$ and such that we have 
\[
v_t(\mu):=(b_t^{(1)}(\mu))^2+(b_t^{(2)}(\mu))^2+((2B)^\gamma-\mu)\left((b_t^{(3)}(\mu))^2+(a_t^{(4)}(\mu))^2\right).
\]
This means that if we define $a_t^{(j)}(\lambda)=b_t^{(j)}(\lambda^\gamma)$, then the $a_t^{(j)}$ are polynomials in $\lambda^\gamma$ of degree $\le t$ for $j\in\{1,2\}$ and of degree $\le t-1$ for $j\in\{3,4\}$ that satisfy
\eqref{e:existence_aj2}. The fact that the $a_t^{(j)}$ are polynomials in $\lambda^\gamma$ of degree $\le t$ or $\le t-1$ together with assumption \eqref{e:F} immediately implies the finite-range properties \eqref{e:existence_aj6} and \eqref{e:existence_aj7}.

\emph{Step 4: Pointwise bounds}\\
It remains to verify the pointwise bounds on $w_t$ and the $a^{(j)}_t$. The explicit formulas for $t<1$ can be read off directly from \eqref{e:existence_aj12}, and so we focus on the case $t\ge1$. The bound \eqref{e:existence_aj3} follows from a similar argument as in \cite{B13}. Indeed, one can check that 
\[\lambda^{m-\gamma/2}\left|\frac{\partial^m}{\partial\lambda^m}\arccos\left(1-\left(\frac{\lambda}{2B}\right)^\gamma\right)\right|\le C_{\gamma,m}\] 
(the analogue of \cite[(2.35)]{B13}). Now, one can proceed as in \cite{B13}. Namely, the main term (with $j=0$) in the definition of $w_t(\lambda)$ can be treated with the same argument as in \cite{B13}, while the terms with $j\ge1$ have even better decay for $t\ge1$.

Next, we will use the identity \eqref{e:existence_aj2} to deduce the bounds on the $a^{(j)}_t$ from \eqref{e:existence_aj3}.  
For $a_t^{(1)}$ and $a_t^{(2)}$ this is straightforward. Namely the estimate \eqref{e:existence_aj4} for $j\in\{1,2\}$ follows directly from \eqref{e:existence_aj2} and the fact that $|a_t^{(j)}(\lambda)|\le \sqrt{w_t(\lambda)}$ for $j\in\{1,2\}$. For $a_t^{(3)}$ and $a_t^{(4)}$ we use that for $\lambda\le B$ the factor $(2B)^\gamma-\lambda^\gamma$ is bounded away from zero, so that we have $|a_t^{(j)}(\lambda)|\le C_B\sqrt{w_t(\lambda)}$ for $j\in\{3,4\}$. This implies again \eqref{e:existence_aj4}.
\end{proof}

Using Lemma \ref{l:existence_aj}, we can now show the other half of Theorem \ref{t:mainthm} (where we assume \eqref{e:Pstar} and \eqref{e:F}).

\begin{proof}[Proof of Theorem \ref{t:mainthm} (ii)]
We define $Q_t=t^{(2-\gamma)/(2\gamma)}\left(a_t^{(1)}(L),a_t^{(2)}(L),Ra_t^{(3)}(L),Ra_t^{(4)}(L)\right)$, where the $a_t^{(j)}$ are as in Lemma \ref{l:existence_aj}. From \eqref{e:existence_aj6} and \eqref{e:existence_aj7} we conclude that $Q_t$ has range at most $\max(\theta(n),\theta(n-1)+\theta_R)$. Next note that $R^*R=(2B)^\gamma\Id-L^\gamma$ commutes with $L$ and thus also with any $F(L)$ for a Borel-measurable function $F$. So from \eqref{e:existence_aj2} we obtain that
\begin{align*}
&w_t(L)\\
&=(a_t^{(1)}(L))^2+(a_t^{(2)}(L))^2+R^*R\left((a_t^{(3)}(L))^2+(a_t^{(4)}(L))^2\right)\\
&=(a_t^{(1)}(L))^*(a_t^{(1)}(L))+(a_t^{(2)}(L))^*(a_t^{(2)}(L))+(R(a_t^{(3)}(L))^*(R(a_t^{(3)}(L))+(R(a_t^{(4)}(L))^*(R(a_t^{(4)}(L))\\
&=\frac{Q_t^*Q_t}{t^{(2-\gamma)/\gamma}}.
\end{align*}
In conbination with \eqref{e:existence_aj1} this implies \eqref{e:finite_range_mainthmH}. The bound on the operator norm of $Q_t$ follows immediately from 
\[\|Q_tu\|_{L^2(X,\HH)}^2=(Q_tu,Q_tu)_{L^2(X,\HH)}=(u,Q_t^*Q_tu)_{L^2(X)}=(u,t^{(2-\gamma)/\gamma}w_t(L)u)_{L^2(X)}\]
and \eqref{e:existence_aj3}.

Next we show the existence of a density for $Q_t$ satisfying the pointwise bound \eqref{e:estimate_kernel2} when we additionally assume \eqref{e:H} and \eqref{e:R}. Obviously we can define this density componentwise. For the first two components the argument is almost the same as in \cite{B13}, and our main task will be to show that the extra operators $R$ in the third and fourth component can be dealt with. We begin, however, by reviewing the argument from \cite{B13} for the first two components.

First of all, we can assume $t\ge1$ throughout, as if $t<1$ we have the explicit formula for $Q_t$.

We write $Q_t^{(j)}=t^{(2-\gamma)/(2\gamma)}a_t^{(j)}(L)$ for $j\in\{1,2\}$ and $Q_t^{(j)}=t^{(2-\gamma)/(2\gamma)}Ra_t^{(j)}(L)$ for $j\in\{3,4\}$. Let $C_b(X)$ be the space of bounded continuous functions on $X$, and let $M(X)\subset C_b(X)^*$ be the space of signed finite Radon measures.

In \cite[Proof of Theorem 1.1]{B13} it is shown that \ref{e:H} implies that $\e^{-tL}\colon M(X)\to L^2(X)$ is continuous (with respect to the weak-* topology on $M(X)$), and that the same also holds for $(1+t^{2/\gamma}L)^{-l}$ for any $l>\frac{\alpha}{4}$. There also the bound
\begin{equation}\label{e:mainthm1}\|(1+t^{2/\gamma}L)^{-l}\delta_x\|_{L^2(X)}\le \frac{C_{l,\alpha,\gamma}\sqrt{\omega(x)}}{t^{\alpha/(2\gamma)}}
\end{equation}
for $l>\frac{\alpha}{4}$ is shown.

For $j\in\{1,2\}$ and $x,y\in X$ we can write
\begin{equation}\label{e:mainthm2}(\delta_x,Q_t^{(j)}\delta_y)_{L^2(X)}=(\delta_x,\sqrt{t}a_t^{(j)}(L)\delta_y)_{L^2(X)}=t^{(2-\gamma)/(2\gamma)}\left(\sqrt{a_t^{(j)}}(L)\delta_x,\sqrt{a_t^{(j)}}(L)\delta_y\right)_{L^2(X)}.
\end{equation}
The estimates \eqref{e:existence_aj4} and \eqref{e:mainthm1} (used for some fixed $l>\frac{\alpha}{4}$) imply that $\sqrt{a_t^{(j)}}(L)\colon M(X)\to L^2(X)$ is continuous and that 
\begin{equation}\label{e:mainthm3}
\left\|\sqrt{a_t^{(j)}}(L)\delta_x\right\|_{L^2(X)}\le  \frac{C_{\alpha,\gamma,B}\sqrt{\omega(x)}}{t^{\alpha/(2\gamma)}}.
\end{equation}
Together with \eqref{e:mainthm2} this immediately implies that $q_t^{(j)}\colon X\times X\to \R$ defined by $q_t^{(j)}(x,y)=(\delta_x,Q_t^{(j)}\delta_y)_{L^2(X)}$ is a continuous kernel for $Q_t^{(j)}$, and that
\begin{equation}\label{e:mainthm4}|q_t^{(j)}(x,y)|\le \frac{C_{\alpha,\gamma,B}\sqrt{\omega(x)\omega(y)}}{t^{(2\alpha+\gamma-2)/(2\gamma)}}
\end{equation}
for $j\in\{1,2\}$ (in line with \eqref{e:mainthm0}).

Finally, we turn to the case $j\in\{3,4\}$. We can write
\[Q_t^{(j)}=t^{(2-\gamma)/(2\gamma)}Ra_t^{(j)}(L)=t^{(2-\gamma)/(2\gamma)}\left(R(1+L)^{-l_R}\right)\left((1+L)^{l_R}a_t^{(j)}(L)\right).\]
The first operator on the right-hand side is well-understood by assumption \eqref{e:R}. For the second operator we can argue similarly as before. Namely we have
\begin{equation}\label{e:mainthm5}
\left(\delta_x,(1+L)^{l_R}a_t^{(j)}(L)\delta_y\right)_{L^2(X)}=\left((1+L)^{l_R/2}\sqrt{a_t^{(j)}}(L)\delta_x,(1+L)^{l_R/2}\sqrt{a_t^{(j)}}(L)\delta_y\right)_{L^2(X)}.
\end{equation}
As before, using \eqref{e:existence_aj4} and \eqref{e:mainthm1} (and the fact that $t\ge1$) we find that $(1+L)^{l_R/2}\sqrt{a_t^{(j)}}(L)\colon M(X)\to L^2(X)$ is continuous and that
\[\left\|(1+L)^{l_R/2}\sqrt{a_t^{(j)}}(L)\delta_x\right\|_{L^2(X)}\le \frac{C_{\alpha,\gamma,B}\sqrt{\omega(x)}}{t^{\alpha/(2\gamma)}}.\]
 
Together with \eqref{e:mainthm5} we conclude that $(1+L)^{l_R}a_t^{(j)}(L)$ has a continuous kernel $\tilde q_t^{(j)}\colon X\times X\to\R$ and that
\[|\tilde q_t^{(j)}(x,y)|\le \frac{C_{\alpha,\gamma,B,l_R}\sqrt{\omega(x)\omega(y)}}{t^{\alpha/\gamma}}.\]
The same argument also shows that $y\mapsto \tilde q_t^{(j)}(\cdot,y)$ is continuous as a map from $X$ to $L^\infty(X)$, i.e. an element of $C_b(X,L^\infty(X))$.

Now we are almost done. From assumption \eqref{e:R} we know that $R(1+L)^{-l_R}$ is given by convolution with the continuous kernel $r\colon X\times X\to \HH_0$ and that $x\mapsto r(x,\cdot)$ is an element of $C_b(X,L^1(X,\HH_0))$. We now define 
\[q_t^{(j)}(x,y)=\int_X r(x,z)\tilde q_t^{(j)}(z,y)\ud\mu(z).\]
By Fubini's theorem (whose use is justified because $r(x,\cdot)\tilde q_t^{(j)}(\cdot,y)\in L^1(X,\HH_0)$) this is a kernel for $Q_t^{(j)}$ and we have the bound
\begin{equation}\label{e:mainthm6}\left|q_t^{(j)}(x,y)\right|_{\HH_0}\le t^{(2-\gamma)/(2\gamma)}\left\|r(x,\cdot)\right\|_{L^1(X,\HH_0)}\left\|\tilde q_t^{(j)}(\cdot,y)\right\|_{L^\infty(X)}\le \frac{C_{\alpha,\gamma,B,l_R,R}\sup_x\omega(x)}{t^{(2\alpha+\gamma-2)/(2\gamma)}}
\end{equation}
for $j\in\{3,4\}$.
The estimates \eqref{e:mainthm4} and \eqref{e:mainthm6} together imply \eqref{e:estimate_kernelH}. 

The estimate \eqref{e:estimate_kernelH2} follows from a similar, but much easier argument. Namely we can write
\[\|Q_t\delta_x\|_{L^2(X,\HH)}^2=(\delta_x,Q_t^*Q_t\delta_x)_{L^2(X)}=(\delta_x,t^{(2-\gamma)/\gamma}w_t(L)\delta_x)_{L^2(X)}=t^{(2-\gamma)/\gamma}\|\sqrt{w_t}(L)\delta_x\|_{L^2(X)}^2\]
and using \eqref{e:existence_aj3} and \eqref{e:mainthm1} we obtain the desired estimate.
\end{proof}

\subsection{Application to Gaussian fields}\label{s:fields}

It remains to explain how Theorem \ref{t:mainthm} can be used to show that various Gaussian fields of interest are in $\F_{\boldsymbol\alpha}$ for an appropriate ${\boldsymbol\alpha}$.

Let us first remind the reader of the definitions of these fields. 
We begin with the discrete fields, and take $X=\Z^d$. We let $\Deltad u(x)=\sum_{y\sim x}(u(y)-u(x))$ be the discrete Laplacian\footnote{In \cite{M22} the alternative normalization $\Deltad u(x)=\frac{1}{2d}\sum_{y\sim x}(u(y)-u(x))$ is used. These normalizations have no effect on the resulting fields except scaling them by a constant factor. We prefer the present normalization as it is most natural from a PDE point of view.}, and let $G^{-\Deltad}$ be the Green's function of $-\Deltad$, i.e. the function $G^{-\Deltad}\colon\Z^d\times\Z^d\to\R$ satisfying $-\Deltad G^{-\Deltad}(\cdot,y)=\delta_y$ and $G^{-\Deltad}(x,y)\to 0$ as $|x-y|\to\infty$. For $d\ge3$ there is a unique such Green's function, and the discrete Gaussian free field can be defined as the unique centered Gaussian field on $\Z^d$ such that \[\Cov(f(x),f(y))=G^{-\Deltad}(x,y)\quad\forall x,y\in\Z^d.\]

The discrete membrane model is defined similarly. We let $(\Deltad)^2$ be the discrete Bilaplacian and $G^{(\Deltad)^2}$ be its Green's function. Then the discrete membrane model is the unique centered Gaussian field on $\Z^d$ such that \[\Cov(f(x),f(y))=G^{(\Deltad)^2}(x,y)\quad\forall x,y\in\Z^d.\]

For the continuous Gaussian free field things are slightly more subtle, as these fields cannot be defined in a pointwise sense (at least when $d\ge2$ or $d\ge4$, respectively). However, they can be defined as random tempered distributions. We take $X=\R^d$, and let $\Delta$ be the standard continuous Laplacian. We denote by $\Sc(\R^d)$ the space of Schwartz functions on $\R^d$, and by $\Sc'(\R^d)$ its dual, the space of tempered distributions. We let $G^{-\Delta}$ be the Green's function of $-\Delta$ (which for $d\ge3$ is uniquely characterized by the condition that $G^{-\Delta}(x,y)\to 0$ as $|x-y|\to\infty$). Then the continuous Gaussian free field is the unique random element of $\Sc'(\R^d)$ such that
\begin{align*}
\E((f,u))&=0\quad\forall u\in\Sc(\R^d),\\
\Cov((f,u),(f,v))&=\int_{\R^d}\int_{R^d}G^{-\Delta}(x,y)u(x)v(y)\ud x\ud y\quad\forall u,v\in\Sc(\R^d).
\end{align*}
The continuous membrane model is again defined similarly. We let $\Delta^2$ be the Bilaplacian, $G^{\Delta^2}$ be its Green's function, and define the continuous membrane model as the unique random element of $\Sc'(\R^d)$ such that
\begin{align*}
\E((f,u))&=0\quad\forall u\in\Sc(\R^d),\\
\Cov((f,u),(f,v))&=\int_{\R^d}\int_{R^d}G^{\Delta^2}(x,y)u(x)v(y)\ud x\ud y\quad\forall u,v\in\Sc(\R^d).
\end{align*}

Having defined these fields, we can now show that they are indeed in $\F$.
\begin{proof}[Proof of Theorem \ref{t:fields_in_Falpha}]$ $

\emph{Step 1: Discrete Gaussian free field}\\
We begin with the case of the discrete Gaussian free field. Thus we consider $\gamma=1$, $X=\Z^d$ for $d\ge3$, the quadratic forms $E(u,v)=\sum_{x\in\Z^d}u(x)(-\Deltad v)(x)$ and $\Phi(u,v)=\sum_{x,y\in\Z^d}u(x)G^{-\Deltad}(x,y)v(y)$, and the operator $L=-\Deltad$. In this setting, we clearly have \eqref{e:Pstar} for any $B\ge4d$ and $\theta(n)=n$, and \eqref{e:H} with $\alpha=d$, and $\omega$ equal to some constant $C$. In order to check \eqref{e:F}, observe first that by polarization it suffices to consider the case $u=v$. We take $B=4d$ and write
\begin{align*}
2B(u,u)_{L^2(X)}-E(u,u)&=8d\sum_{x\in\Z^d}u(x)^2-\sum_{x\in\Z^d}\sum_{i=1}^d(u(x+e_i)-u(x))^2\\
&=\sum_{x\in\Z^d}\left(4du(x)^2+\sum_{i=1}^d2u(x)^2+2u(x+e_i)^2-(u(x+e_i)-u(x))^2\right)\\
&=\sum_{x\in\Z^d}\left(4du(x)^2+\sum_{i=1}^d(u(x+e_i)+u(x))^2\right)
\end{align*}
where $e_1,\ldots,e_d$ are the standard unit vectors in $\R^d$.
Hence, choosing $\HH_0=\R^{d+1}$ and defining $Ru(x):=\left(2\sqrt{d}u(x),u(x+e_1)+u(x),\ldots,u(x+e_d)+u(x)\right)$, we obtain \eqref{e:F} with $\theta_R=1$ \footnote{Note that $B\Id-L$ could even be written as $\hat R^*\hat R$ for $\hat Ru(x):=\left(u(x+e_1)+u(x),\ldots,u(x+e_d)+u(x)\right)$, i.e. for an $\R^d$-valued operator. This might allow to take $\HH_0=\R^d$ instead of $\HH_0=\R^{d+1}$ here. However, in the proof of Lemma \ref{l:existence_aj} we needed to work with the factor $((2B)^\gamma-\lambda^\gamma)$ instead of $(B^\gamma-\lambda^\gamma)$, as otherwise it is unclear how to obtain sharp estimates on the $a^{(j)}_t$ (cf. the very last sentence of the proof of the lemma). In other words, if one wanted to work with $\hat R$ instead of $R$, it would no longer be clear if the resulting $q_t$ have the optimal decay rate in $t$.}. The regularity assumption \eqref{e:R} clearly holds with $l_R=0$.

We can now apply Theorem \ref{t:mainthm} (ii), and obtain a family $Q_t$ of linear maps $L^2(X)\to L^2(X,\R^{2d+4})$ with associated densities $q_t\colon X\times X\to\R^{2d+4}$ satisfying \eqref{e:finite_range_mainthmH} and \eqref{e:estimate_kernelH}. As $-\Deltad$ is translation-invariant, $q_t$ is translation-invariant (cf. the discussion in Section \ref{s:discussion}), and so, slightly abusing notation, we write $q_t(x-y)$ for $q_t(x,y)$.

The density $q_t$ is $\R^{2d+4}$-valued, while we are looking for a $\R$-valued density. We can achieve this by letting $\q(\cdot,t)$ cycle through the $2d+4$ components of $q_t$. More precisely, we define a precursor $\tilde \q$ to $\q$ piecewise by setting
\[\tilde\q(\cdot,t)=(2d+4)e_j\cdot q_{n+(2d+4)\left(t-n-\frac{j-1}{2d+4} \right)} \text{ for }t\in \left[n+\frac{j-1}{2d+4},n+\frac{j}{2d+4}\right),n\in\N,j\in\{1,2,\ldots,2d+4\}\]
where (with a slight abuse of notation) $e_1,\ldots,e_{2d+4}$ are the standard unit vectors in $\R^{2d+4}$.
Note that for each fixed $j$ and $n$ we have 
\[\left\{n+(2d+4)\left(t-n-\frac{j-1}{2d+4}\right)\colon t\in \left[n+\frac{j-1}{2d+4}\right)\right\}=[n,n+1)\]
and so $\tilde\q$ "sees" all the components of $q_t$.

Indeed, the relation \eqref{e:finite_range_mainthmH} implies
\begin{equation}\label{e:fields_in_Falpha1}
\begin{split}
G^{-\Deltad}(x,y)&=\Phi(\delta_x,\delta_y)\\
&=\int_0^\infty(Q_t(\delta_x),Q_t(\delta_y))_{L^2(X,\R^{2d+4})}\ud t\\
&=\int_0^\infty\sum_{j=1}^{2d+4}(e_j\cdot q_t(x-\cdot),e_j\cdot q_t(y-\cdot))_{L^2(X)}\\
&=\sum_{n=0}^\infty\int_n^{n+1}\sum_{j=1}^{2d+4}(e_j\cdot q_t(x-\cdot),e_j\cdot q_t(y-\cdot))_{L^2(X)}\\
&=\sum_{n=0}^\infty\int_n^{n+1}(\tilde \q(x-\cdot,t),\tilde \q(y-\cdot,t))_{L^2(X)}\\
&=\int_0^\infty(\tilde \q(x-\cdot,t),\tilde \q(y-\cdot,t))_{L^2(X)}.
\end{split}
\end{equation}
Furthermore, $\tilde \q$ inherits the finite-range property from $q_t$. One easily checks that $\supp\tilde \q\subset\left\{(x,t)\colon |x|\le\lceil t\rceil\right\}$. This is almost property (iv) in Definition \ref{d:Falpha}. We now define
\[\q(x,t)=\begin{cases}0&t\le 1\\ \frac12\tilde \q\left(x',(t-1)/2\right)&t>1,x'\in\Z^d,x\in x'+\left[-\frac12,\frac12\right)^d\end{cases}\]
Then $\q$ does satisfy Property (iv) in Definition \ref{d:Falpha}, and \eqref{e:fields_in_Falpha1} implies that the field $f:=\q*_1 W\restriction_{\Z^d}$ has the law of the discrete Gaussian free field. 
It remains to check the other properties in Definition \ref{d:Falpha}. Property (i) follows from the symmetry of $-\Deltad$, (iii) is obvious. Furthermore,
according to \eqref{e:estimate_kernelH2} (with $\alpha=d$) we have
\[\|q_t\|_{L^2(X,\R^{2d+4})}\le\frac{C}{t^{(d-1)/2}}\]
for $t>1$ and $q_t=C_B\Id$ for $t\le1$. This means that also
\[\|\q(\cdot,t)\|_{L^2(\R^d)}\le\frac{C_d}{t^{(d-1)/2}}\]
and so we easily obtain Property (v) as well. This completes the verification that the discrete Gaussian free field is in class $\F_{d-2}$.

\emph{Step 2: Discrete membrane model}\\
The proof of Theorem \ref{t:fields_in_Falpha} for the discrete membrane model is very similar. This time we take $\gamma=\frac12$ and consider $E(u,v)=\sum_{x\in\Z^d}u(x)(\Deltad)^2v(x)$ and $L=(\Deltad)^2$.
We take $B=16d^2$ and note that a calculation very similar to the one before shows that
\[
\sqrt{2B}(u,u)_{L^2(X)}-(u,-\Deltad u)=\sum_{x\in\Z^d}\left(4(\sqrt{2}-1)du(x)^2+\sum_{i=1}^d(u(x+e_i)+u(x))^2\right)
\]
and so we can take $Ru(x):=\left(2\sqrt{(\sqrt{2}-1)d}u(x),u(x+e_1)+u(x),\ldots,u(x+e_d)+u(x)\right)$, $\HH_0=\R^{d+1}$, $\theta_R=1$. Moreover we have the heat kernel bound \eqref{e:H} with $\alpha=\frac{d}{2}$, and \eqref{e:R} holds with $l_R=1$. Proceeding now as before, we see that the discrete membrane model is indeed in $\F_{d-4}$.

\emph{Step 3: Continuous Gaussian free field}\\
The proofs for the continuous fields are in some sense easier, as we no longer need to find a factorization as in \eqref{e:F}. On the other hand, we now need extra care as the relevant fields are no longer defined in a pointwise sense.

We let $X=\R^d$. It is clear that $L=-\Delta$ satisfies assumptions \eqref{e:P} and \eqref{e:H} with $\gamma=1$ and $\alpha=d$. So from Theorem \ref{t:mainthm} we obtain a family $Q_t$ of linear maps $L^2(X)\to L^2(X)$ with associated densities $q_t\colon X\times X\to\R$ such that \eqref{e:estimate_kernel} and \eqref{e:estimate_kernelH} hold. Again, the $q_t$ are translation-invariant, and we write $q_t(x-y)$ for $q_t(x,y)$.  The finite-range property \eqref{e:finite_range_mainthm} implies in particular that
\begin{equation}\label{e:fields_in_Falpha2}
\begin{split}
Cov((f,u),(f,v))&=\int_{\R^d}\int_{R^d}G^{-\Delta}(x,y)u(x)u(y)\ud x\ud y\\
&=\int_0^\infty (q_t*u)(z)(q_t*v)(z)\ud z \ud t
\end{split}
\end{equation}
for $u,v\in\Sc(\R^d)$.

We are actually looking for a decomposition not of $f$ but of its mollified version $\eta*f$. It turns out than we can just take $\eta*q_t$ instead of $q_t$ to achieve this. Indeed, one easily checks that
\[\Cov((\eta*f,u),(\eta*f,v))=\Cov((f,\eta*u),(f,\eta*v))\]
(here we used that $\eta$ is symmetric around 0, otherwise we would need to take $\eta(-\cdot)$ on the right-hand side), and using \eqref{e:fields_in_Falpha1} and the commutativity and associativity of convolution we can rewrite this as
\begin{align*}
\Cov((\eta*f,u),(\eta*f,v))&=\int_0^\infty (q_t*(\eta*u))(z)(q_t*(\eta*v))(z)\ud z \ud t\\
&=\int_0^\infty ((\eta*q_t)*u)(z)((\eta*q_t)*v)(z)\ud z \ud t.
\end{align*}
This means that if we set 
\[\tilde\q(x,t)=\eta*q_t(x)=\int_{\R^d}\eta(x-y)q_t(y)\ud y\]
then $\eta*f\overset{\text{law}}{=}\tilde\q*_1 W$. Moreover, as $\eta$ has finite support, the function $\tilde\q$ has the finite-range property $\supp\tilde \q\subset\left\{(x,t)\colon |x|\le t+\diam\supp\eta\right\}$.

Now similarly as in Step 1, we set
\[\q(x,t)=\begin{cases}0&t\le \diam\supp\eta\\ \frac12\tilde \q\left(x,(t-\diam\supp\eta)/2\right)&t>\diam\supp\eta\end{cases}.\]
Then $\q$ satisfies Property (iv) in Definition \ref{d:Falpha}, and we still have that $\eta*f\overset{\text{law}}{=}\q*_1 W$. Properties (i), (ii) and (iii) are obvious from the construction, and for Property (v) we observe that \eqref{e:estimate_kernelH} implies 
\[\|q_t\|_{L^2(X)}\le\frac{C}{t^{(d-1)/2}}\]
and hence also
\[\|\q(\cdot,t)\|_{L^2(\R^d)}\le\frac{C}{t^{(d-1)/2}}\]
which easily implies the desired estimate. Thus the field is indeed in $\F_{d-2}$.

\emph{Step 4: Continuous membrane model}\\
The argument is completely analogous to the one in Step 3, and so we omit the details.
\end{proof}

\subsubsection*{Acknowledgements} 
The author would like to thank Roland Bauerschmidt, Stephen Muirhead and an anomymous referee for various helpful comments that helped improve the manuscript.

The author was supported by the Foreign Postdoctoral Fellowship Program of the Israel Academy of Sciences and Humanities, and partially by Israel Science Foundation grant number 421/20.

\paragraph{Declarations of interest:} none.

\end{document}